\documentclass[12pt,a4paper,usenames]{amsart}
\usepackage{amssymb,amsmath}
\usepackage[mathcal]{eucal}
\usepackage{tikz}

\usepackage{hyperref}

\hypersetup{
    colorlinks=true,
    linkcolor=magenta,
    filecolor=cyan,      
    urlcolor=magenta,
    citecolor=magenta
}

\theoremstyle{plain} 
\newtheorem{theorem}{Theorem}[section]
\newtheorem{lemma}[theorem]{Lemma}
\newtheorem{corollary}[theorem]{Corollary}
\newtheorem{proposition}[theorem]{Proposition}

\theoremstyle{definition}
\newtheorem{definition}[theorem]{Definition}
\newtheorem{remark}[theorem]{Remark}

\newtheorem*{ackno}{Acknowledgement}
\newtheorem*{condition}{`Lifting Condition'}

\numberwithin{equation}{section}

\newcommand{\N}{\mathbb{N}}
\newcommand{\Z}{\mathbb{Z}}
\newcommand{\Q}{\mathbb{Q}}
\newcommand{\C}{\mathbb{C}}

\newcommand{\GL}{\mathsf{GL}}
\newcommand{\aGm}{\mathsf{G}_\mathrm{m}}
\newcommand{\aGa}{\mathsf{G}_\mathrm{a}}

\newcommand{\aG}{\mathsf{G}}
\newcommand{\aH}{\mathsf{H}}
\newcommand{\aN}{\mathsf{N}}
\newcommand{\aT}{\mathsf{T}}

\newcommand{\mcalB}{\mathcal{B}}

\newcommand{\fk}{\mathcal{K}}
\newcommand{\fo}{\mathcal{O}}
\newcommand{\fp}{\mathfrak{p}}

\newcommand{\ee}{\mathbf{e}}
\newcommand{\ff}{\mathbf{f}}
\renewcommand{\gg}{\mathbf{g}}

\newcommand{\coloneqq}{\mathrel{\mathop:}=}

\DeclareMathOperator{\Aut}{Aut} 
\DeclareMathOperator{\Hom}{Hom} 
\DeclareMathOperator{\Rad}{Rad}
\DeclareMathOperator{\Mat}{Mat}
\DeclareMathOperator{\diag}{diag}

\renewcommand{\phi}{\varphi} 
\renewcommand{\rho}{\varrho} 
\renewcommand{\theta}{\vartheta}

\begin{document}

\title[A family of class-$2$ nilpotent groups]{A family of class-$2$
  nilpotent groups, their automorphisms and pro-isomorphic zeta
  functions}

\author{Mark N.\ Berman}\address{Department of Mathematics, Ort Braude
  College, P.O. Box 78, Snunit St., 51, Karmiel 2161002, Israel}
\email{berman@braude.ac.il}

\author{Benjamin Klopsch} \address{Mathematisches Institut,
  Heinrich-Heine-Universit\"at D\"usseldorf, 40225 D\"usseldorf, Germany}
\email{klopsch@math.uni-duesseldorf.de}

\author{Uri Onn} \address{Department of Mathematics, Ben Gurion
  University of the Negev, Beer-Sheva 84105, Israel}
\email{urionn@math.bgu.ac.il}



\maketitle

\begin{abstract}
  The pro-isomorphic zeta function $\zeta^\wedge_\Gamma(s)$ of a
  finitely generated nilpotent group $\Gamma$ is a Dirichlet
  generating function that enumerates finite-index subgroups whose
  profinite completion is isomorphic to that of~$\Gamma$.  Such zeta
  functions can be expressed as Euler products of $p$-adic integrals
  over the $\mathbb{Q}_p$-points of an algebraic automorphism group
  associated to $\Gamma$.  In this way they are closely related to
  classical zeta functions of algebraic groups over local fields.

  We describe the algebraic automorphism groups for a natural family
  of class\nobreakdash-$2$ nilpotent groups $\Delta_{m,n}$ of Hirsch
  length $\binom{m+n-2}{n-1}+\binom{m+n-1}{n-1} + n$ and central
  Hirsch length~$n$; these groups can be viewed as generalisations of
  $D^*$-groups of odd Hirsch length.  General $D^*$-groups, that is
  `indecomposable' finitely generated, torsion-free class-$2$
  nilpotent groups with central Hirsch length~$2$, were classified up
  to commensurability by Grunewald and Segal. 

  We calculate the local pro-isomorphic zeta functions for the groups
  $\Delta_{m,n}$ and obtain, in particular, explicit formulae for the
  local pro-isomorphic zeta functions associated to $D^*$-groups of
  odd Hirsch length.  From these we deduce local functional equations;
  for the global zeta functions we describe the abscissae of
  convergence and find meromorphic continuations.  We deduce that the
  spectrum of abscissae of convergence for pro-isomorphic zeta
  functions of class-$2$ nilpotent groups contains infinitely many
  cluster points.  For instance, the global abscissa of convergence of
  the pro-isomorphic zeta function of a $D^*$-group of Hirsch length
  $2m+3$ is shown to be $6-\frac{15}{m+3}$.
\end{abstract}

\section{Introduction}\label{section:introduction}

Zeta functions provide a compact and powerful way to encode
information about the lattice of finite-index subgroups of a finitely
generated group.  Our attention focuses on finitely generated,
torsion-free nilpotent groups, for which a rich theory is available;
see~\cite{dSWo08} and references therein.  In the current paper we are
interested in the \emph{pro-isomorphic zeta function} of such a group
$\Gamma$, i.e., the Dirichlet generating function
\begin{equation*}
  \zeta_\Gamma^\wedge(s) =  \sum_{n=1}^\infty \frac{a_n^\wedge(\Gamma)}{n^s},
\end{equation*}
where $a_n^\wedge(\Gamma)$ denotes the number of subgroups $\Delta$ of
index $n$ in $\Gamma$ such that the profinite completion
$\widehat{\Delta}$ is isomorphic to the profinite completion
$\widehat{\Gamma}$ of the ambient group. 

It follows from the nilpotency of $\Gamma$ that one obtains an Euler
product decomposition over all rational primes $p$:
\begin{equation} \label{equ:local-factor} \zeta_\Gamma^\wedge(s) =
  \prod_p \zeta_{\Gamma,p}^\wedge(s), \qquad \text{where} \quad
  \zeta_{\Gamma,p}^\wedge(s) = \sum_{k=0}^\infty
  a_{p^k}^\wedge(\Gamma) p^{-ks}
\end{equation}
is called the \emph{local zeta function} at a prime~$p$;
see~\cite{GrSeSm88}.  Furthermore, each of the local factors
$\zeta_{\Gamma,p}^\wedge(s)$ is a rational function over $\Q$ in
$p^{-s}$.

In comparison to other zeta functions of groups, a unique feature of
pro-isomorphic zeta functions $\zeta^\wedge_{\Gamma}(s)$ is their
relation to a rather different object of independent interest.  Let
$\fk$ be a number field with ring of integers $\fo$, and let $\aG
\subseteq \GL_d$ be an affine group scheme over $\fo$ with a fixed
embedding into $\GL_d$.  For a finite prime $\fp$, let $\fk_{\fp}$
denote the completion at $\fp$, and let $\fo_\fp$ denote the valuation
ring of~$\fk_\fp$.  Putting $\aG_\fp = \aG(\fk_\fp)$, $\aG^+_\fp =
\aG(\fk_\fp) \cap \Mat_d(\fo_\fp)$ and $\aG(\fo_\fp) = \aG(\fk_\fp)
\cap \GL_d(\fo_\fp)$, one defines the zeta function of $\aG$ at $\fp$
as
\begin{equation}\label{equ:zeta-alg-group}
  \mathcal{Z}_{\aG,\fp}(s) = \int_{\aG^+_\fp} \lvert \det(g)
  \rvert_\fp^s \, d\mu_{\aG_\fp}(g),
\end{equation}
where $|.|_\fp$ is the $\fp$-adic absolute value on $\fk_\fp$ and
$\mu_{\aG_\fp}$ is the right Haar measure on $\aG(k_\fp)$, normalised
so that $\mu_{\aG_\fp}(\aG(\fo_\fp))=1$.  In the past century,
$\mathcal{Z}_{\aG,\fp}(s)$ was studied by Hey, Weil, Tamagawa, Satake,
Macdonald and Igusa~\cite{He29,We62,Ta63,Sa63,Ma79,Ig89}, for
independent reasons.  Grunewald, Segal and Smith~\cite{GrSeSm88}
recognised that these classical zeta functions of algebraic groups
relate to pro-isomorphic zeta functions of nilpotent groups; see
Section~\ref{subsec:Lie-corr}.  Studying the pro-isomorphic zeta
function of a finitely generated nilpotent group typically involves
two steps: first one needs to understand an associated algebraic
automorphism group $\aG$ which comprises an affine group scheme
over~$\Z$, and subsequently one studies the $\fp$-adic
integral~\eqref{equ:zeta-alg-group}.

\subsection{Main results} \label{sec:groups-Delta}

In this paper we consider a certain family of finitely generated,
torsion-free class\nobreakdash-$2$ nilpotent groups~$\Delta_{m,n}$,
defined below, and set about explicitly computing the associated
algebraic automorphism groups and local pro-isomorphic zeta functions.

Our original interest concerned finitely generated, torsion-free
class-$2$ nilpotent groups of central Hirsch length~$2$; we refer to
such groups as $D^*$-groups.  Up to commensurability, $D^*$-groups
were classified by Grunewald and Segal~\cite{GrSe84} in terms of
indecomposable constituents of their radicable hulls; the latter are
called $\mathcal{D}^*$-groups and our terminology is a natural
adaptation.  For any integer $m \geq 1$, there is a unique
commensurability class of `indecomposable' $D^*$-groups of Hirsch
length $2m+3$, represented by the group 
\[
\Gamma_m = \langle a_1, \ldots, a_m, b_1, \ldots, b_{m+1}, c_1,
c_2 \mid R \rangle,
\]
where the finite set of relations $R$ specifies that $c_1, c_2$ are
central and
\[ [a_i,a_j]= 1, \quad [b_i,b_j]= 1, \quad [a_i,b_j]=
\begin{cases}
c_1 & \text{if $i=j$,} \\
c_2 & \text{if $i+1=j$,} \\
1   & \text{otherwise,}
\end{cases} 
\qquad \text{for all relevant indices.}
\]

For these $D^*$-groups of odd Hirsch length we obtain the following
results.

\begin{theorem}\label{theorem:dstar-zeta}
  Let $\Gamma_m$ be the `indecomposable' $D^*$-group of Hirsch length
  $2m+3$ defined above, for $m \geq 1$.  Then, for every rational
  prime~$p$,
  \[
  \zeta_{\Gamma_m,p}^\wedge(s) = \frac{1+p^{\frac{1}{2} (9m^2+m-2) -
      (m^2+2m-1)s}}{\big(1 - p^{\frac{1}{2} m(9m+1) - (m^2+2m-1)s}
    \big) \big( 1-p^{4m+2-(m+2)s} \big) \big( 1-p^{6m+2-(m+3)s}
    \big)}.
  \]
\end{theorem}

\begin{remark} We remark that the `indecomposable' $D^*$-group
  $\Gamma_1$ of Hirsch length $5$ is none other than the Grenham group
  of the same Hirsch length whose local pro-isomorphic zeta functions
  have previously been determined: indeed, setting $m=1$ in
  Theorem~\ref{theorem:dstar-zeta} or setting $d=3$ in the formula
  appearing in \cite[Section~3.3.13.2]{Be05}, we obtain the same
  expression
  \[
  \zeta_{\Gamma_1,p}^\wedge(s) = \frac{1}{(1-p^{4-2s})(1-p^{5-2s})(1-p^{6-3s})}.
  \]
\end{remark}

\begin{corollary} \label{cor:abscissa} Let $\Gamma_m$ be the
  `indecomposable' $D^*$-group of Hirsch length $2m+3$ defined above,
  for $m\geq 1$.  Then the pro-isomorphic zeta function
  $\zeta_{\Gamma_m}^\wedge(s)$ satisfies local functional equations
  and admits meromorphic continuation to the entire complex plane.  It
  has abscissa of convergence $3$, if $m=1$, and $6-\frac{15}{m+3}$,
  if $m \geq 2$.
\end{corollary}

The occurrence of local functional equations for pro-isomorphic zeta
functions is a widespread but not ubiquitous phenomenon; cf.\
\cite{BeKl15}.  In the present instance it refers to the fact that
\[
\zeta_{\Gamma_m,p}^\wedge(s) \vert_{p \to p^{-1}} = - p^{10m+5 -
  (2m+5)s} \, \zeta_{\Gamma_m,p}^\wedge(s) \qquad \text{for every
  rational prime $p$},
\]
using the formula in Theorem~\ref{theorem:dstar-zeta}.
The theorem and its corollary significantly widen
the scope of known examples.  In particular, they indicate that there is a
rich spectrum of abscissae of convergence for pro-isomorphic zeta
functions of class-$2$ nilpotent groups; cf.\
\cite[Problem~1.3]{dSGr06}.  

\smallskip

In fact, we study a much larger family of class-$2$ nilpotent groups,
generalising $D^*$-groups of odd Hirsch length in the following way.
Let $m,n \in \N$ with $n \geq 2$.  Put
\begin{align*}
  \mathbf{E} & = \{ \ee \mid \ee = (e_1, \ldots, e_n) \in \N_0^n
  \text{ with $e_1 + \ldots + e_n = m-1$} \}, \\
  \mathbf{F} & = \{ \ff \mid \ff = (f_1,\ldots,f_n) \in \N_0^n \text{
    with $f_1 + \ldots + f_n = m$} \}.
\end{align*}
We consider the group $\Delta_{m,n}$ on $\lvert \mathbf{E} \rvert +
\lvert \mathbf{F} \rvert + n$ generators
\begin{equation*}
  \{ a_\ee \mid \ee \in \mathbf{E} \} \cup \{ b_\ff \mid  \ff \in
  \mathbf{F} \} \cup \{
  c_j \mid  j \in \{1, \ldots, n\} \},
\end{equation*}
subject to the defining relations
\[
[a_\ee,a_{\ee'}] = [b_\ff,b_{\ff'}]= [a_\ee,c_j] = [b_\ff,c_j] =
[c_j,c_{j'}] = 1
\]
for all $\ee, \ee' \in \mathbf{E}$, $\, \ff, \ff' \in \mathbf{F}$, $\, j, j'
\in \{1,\ldots,n\}$ and
\[
[a_\ee,b_\ff] =
\begin{cases} 
  c_i & \text{if $\ff - \ee$ is of the form
    $(\underbrace{0,\ldots,0}_{i-1 \text{
        entries}},1,\underbrace{0,\ldots,0}_{n-i \text{ entries}})$
    for some $i$,} \\
  1 & \text{if $\ff - \ee$ is not of this form,}
\end{cases}
\]
for $\ee \in \mathbf{E}$ and $\ff \in \mathbf{F}$.  Consequently,
$\Delta_{m,n}$ is a finitely generated, torsion-free class-$2$
nilpotent group of Hirsch length
\[
\lvert \mathbf{E} \rvert + \lvert \mathbf{F} \rvert + n =
\binom{m+n-2}{n-1} + \binom{m+n-1}{n-1} + n
\]
whose centre coincides with the commutator subgroup and has Hirsch
length~$n$; indeed,
\[
\mathrm{Z}(\Delta_{m,n}) = [\Delta_{m,n},\Delta_{m,n}] = \langle c_1,
\ldots, c_n \rangle.
\]
We can now state our main result.  For the subsequent analysis, it is
convenient to think of $m$ as a primary parameter, each $m$ yielding
a sequence of groups $\Delta_{m,n}$ for $n \in \mathbb{N}_{\ge 2}$; we
choose to write binomial coefficients in accordance with this.

\begin{theorem}\label{theorem:main}
  Let $\Delta_{m,n}$ be the nilpotent group defined above, for $m \geq
  1$ and $n\geq 2$.  Let $\Phi$ be the set of roots of the algebraic
  group~$\GL_n$, with negative roots $\Phi^-$ determined by some
  choice of simple roots $\alpha_1, \ldots, \alpha_{n-1}$, and let
  $\ell$ denote the standard length function on the Weyl group $W
  \cong \mathrm{Sym}(n)$ acting on $\Phi$.  Then, for all rational
  primes $p$,
  \begin{equation} \label{equ:main-thm}
    \zeta^\wedge_{\Delta_{m,n},p}(s) = \frac{\sum_{w \in W}
      p^{-\ell(w)} \prod_{i=1}^{n-1} X_i^{\,
        \nu_i(w)}}{\prod_{i=1}^{n-1} (1-X_i)} \cdot
    \frac{1}{(1-\widetilde{X}_0)(1-\widetilde{X}_n)},
  \end{equation}
  where
  \[
  \nu_i(w) =
  \begin{cases}
    1 & \text{if $\alpha_i \in w(\Phi^-)$,} \\
    0 & \text{otherwise},
  \end{cases}
  \qquad \text{for $1 \leq i \leq n-1$,}
  \]
  $X_i = p^{A_i - B_i s}$, for $1 \le i \le n-1$, and
  $\widetilde{X}_j = p^{\widetilde{A}_j - \widetilde{B}_j s}$, for $j \in
  \{0,n\}$, with
  \begin{align*}
    A_i & = i(n-i) + \Big( \tbinom{m+n-2}{m-1} + \tbinom{m+n-1}{m}
          \Big) \big( (m-1) n + i \big)
    \\
        & \quad + \sum\nolimits_{j=1}^i \left( 1 +
          \tfrac{(m-1)(i-j+1)}{n-j+1} \right) \tbinom{m+j-2}{m-1} 
          \tbinom{m+n-j-1}{m-1}, \\
    B_i & = - m(m-1) \tbinom{m+n-2}{m} + \Big( 1 +
          \tbinom{m+n-2}{m-1} \Big) \big( (m-1) n + i \big) 
  \end{align*}
  and
  \begin{align*}
    \widetilde{A}_0 & = n \Big( \tbinom{m+n-2}{m-1} + \tbinom{m+n-1}{m}
                      \Big), && \widetilde{A}_n = \widetilde{A}_0 +
                                  \tbinom{2m+n-2}{2m-1},\\
    \widetilde{B}_0 & = \tbinom{m+n-2}{m-1} + n,
                               && \widetilde{B}_n = \widetilde{B}_0 +
                                  \tbinom{m+n-2}{m} = \tbinom{m+n-1}{m} + n.
  \end{align*}
  Furthermore, these parameters are connected by the relations
  \[
  (m-1) \widetilde{A}_0 = A_0, \quad (m-1) \widetilde{B}_0 = B_0,
  \quad m \widetilde{A}_n = A_n, \quad m \widetilde{B}_n = B_n. 
  \]
\end{theorem}
    
We remark that the group $\Delta_{m,2}$ is isomorphic to the
$D^*$-group $\Gamma_m$ discussed earlier and indeed,
Theorem~\ref{theorem:dstar-zeta} is obtained as a special case of
Theorem~\ref{theorem:main} by setting $n=2$. The following can be
deduced by a method of Igusa \cite{Ig89} using symmetries of the Weyl
group $W$; further details are provided in
Section~\ref{sec:corollaries}.

\begin{corollary}\label{corollary:fn-eq}
  Let $\Delta_{m,n}$ be the nilpotent group defined above, for
  $m \geq 1$ and $n\geq 2$.  Then for all primes $p$, the local zeta
  function $\zeta^\wedge_{\Delta_{m,n},p}(s)$ satisfies a functional
  equation of the form
  \begin{equation*}
    \zeta^\wedge_{\Delta_{m,n},p}(s)|_{p\to p^{-1}} =
    (-1)^{n-1}p^{a+bs} \zeta^\wedge_{\Delta_{m,n},p}(s)
  \end{equation*}
  where
  \begin{equation} \label{equ:sym-fac-a-b}
    \begin{split}
      a & = \tbinom{n}{2}+2n \Big( \tbinom{m+n-2}{m-1} +
      \tbinom{m+n-1}{m} \Big) + \tbinom{2m+n-2}{2m-1} ,\\
      b & = (2m-1) \tbinom{m+n-2}{m}-2n \Big( 1 +
      \tbinom{m+n-2}{m-1} \Big).
    \end{split}
  \end{equation}
\end{corollary}

This reduces to the functional equation for
$\zeta_{\Gamma_m,p}^\wedge(s)$ by setting $n=2$. A slight modification
of the main result in \cite{Be11} can be used to deduce functional
equations already half way through the proof of our theorem; see
Remark~\ref{rem:funct-equa-half-way}.  A thorough analysis of the
explicit formulae in Theorem~\ref{theorem:main} yields the following
information about convergence and meromorphic continuation; see
Section~\ref{sec:corollaries}.

\begin{corollary}\label{corollary:mero-cont}
  Let $\Delta_{m,n}$ be the nilpotent group defined above, for
  $m \geq 1$ and $n\geq 2$.  Then the pro-isomorphic zeta function
  $\zeta^\wedge_{\Delta_{m,n}}(s)$ has abscissa of convergence
  \[
  \alpha(m,n) = 
  \begin{cases}
   (A_1 + 1)/ B_1  = n+1 & \text{if $m = 1$,} \\
   (\widetilde{A}_0 +1) / \widetilde{B}_0 & \text{if $(m,n) \in \mathcal{C}$,} \\
   (\widetilde{A}_n +1) / \widetilde{B}_n & \text{otherwise,} \\ 
  \end{cases}
  \]
  where $\widetilde{A}_j, \widetilde{B}_j$ are as in
  Theorem~\textup{\ref{theorem:main}} and
  \[
  \mathcal{C} = (\{2,3\} \times \mathbb{N}_{\ge 3}) \cup (\{4\} \times
  \{4,5,\ldots,38\}) \cup (\{5\} \times \{5,6,7,8,9\}).
  \]

  Moreover, $\zeta^\wedge_{\Delta_{m,n}}(s)$ admits meromorphic
  continuation (at least) to the half-plane consisting of all
  $s \in \C$ with $\mathrm{Re}(s) > \beta(m,n)$,
  where
  \[
  \beta(m,n) = \max\{ A_i /B_i \mid 1 \le i \le n-1\} <
  \alpha(m,n).
  \]
  The resulting pole of $\zeta^\wedge_{\Delta_{m,n}}(s)$ at
  $s=\alpha(m,n)$ is simple.
\end{corollary}

\begin{remark} \label{rem:Grenham} We remark that the group
  $\Delta_{1,n}$ of Hirsch length $2n+1$ is none other than the
  Grenham group of the same Hirsch length, whose local pro-isomorphic
  zeta functions have previously been determined by the first author.
  Setting $m=1$ in Theorem~\ref{theorem:main} or setting $d=n+1$ in
  the formula appearing in~\cite[Section~3.3.13.2]{Be05} we obtain
  different but equivalent expressions for the local pro-isomorphic
  zeta functions of $\Delta_{1,n}$, namely
  \begin{equation} \label{equ:explicit-Grenham-1}
    \zeta_{\Delta_{1,n},p}^\wedge(s) = \frac{\sum_{w \in W} p^{-\ell(w)}
      \prod_{i=1}^{n-1} (p^{i(2n+2-i) -2is})^{
        \nu_i(w)}}{(1-p^{n(n+1)-(n+1)s}) \prod_{i=1}^n ( 1- p^{i(2n+2-i)
        -2is} )},
  \end{equation}
  where $\ell$ denotes the standard length function on
  $W \cong \mathrm{Sym}(n)$, and
  \begin{equation} \label{equ:explicit-Grenham-2}
    \zeta_{\Delta_{1,n},p}^\wedge(s) = \frac{1}{(1-p^{n(n+1)-(n+1)s})
      \prod_{i=1}^n (1-p^{n+1+i - 2s})}.
  \end{equation}

  The formula~\eqref{equ:explicit-Grenham-2} can be derived by a direct
  computation that does not require the general approach via Bruhat
  decompositions; compare ~\cite[Section~3.3.2]{Be05}.  The connection
  between \eqref{equ:explicit-Grenham-1} and
  \eqref{equ:explicit-Grenham-2} is given by the identity
  \[
  \frac{\sum_{w \in W} p^{-\ell(w)} \prod_{i=1}^{n-1} (p^{i(n-i) -
      it})^{\nu_i(w)}}{\prod_{i=1}^n ( 1-p^{i(n-i) -it} )} =
  \int_{\GL_n(\Q_p)^+} \lvert \det B \rvert_p^{\, t} \, d\mu =
  \frac{1}{\prod_{i=1}^n (1 - p^{i-1-t})}
  \]
  applied to $t = 2s-n-2$.  Observe that by virtue
  of~\eqref{equ:explicit-Grenham-2}, the zeta function
  $\zeta_{\Delta_{1,n},p}^\wedge(s)$ admits meromorphic continuation
  to the entire complex plane.
\end{remark}

It would be interesting to study further analytic properties of the
global zeta functions associated to the nilpotent
groups~$\Delta_{m,n}$.  Moreover it would be illuminating to study in
more depth the distribution of the abscissae of convergence
$\alpha(m,n)$ of the groups~$\Delta_{m,n}$; compare
Figure~\ref{fig:abscissae}.  A straightforward analysis yields the
following corollary, which shows that the spectrum of abscissae of
convergence for pro-isomorphic zeta functions of class-$2$ nilpotent groups
contains infinitely many cluster points; cf.\
\cite[Problem~1.3]{dSGr06}.

\begin{corollary} \label{cor:limit-points}
  For each $n\geq 2$, the abscissae $\alpha(m,n)$ of
  $\Delta_{m,n}$ converge, as $m \to \infty$, namely to 
  \[
  \lim_{m\to \infty} \alpha(m,n)=2n+2^{n-1}.
  \]
\end{corollary}

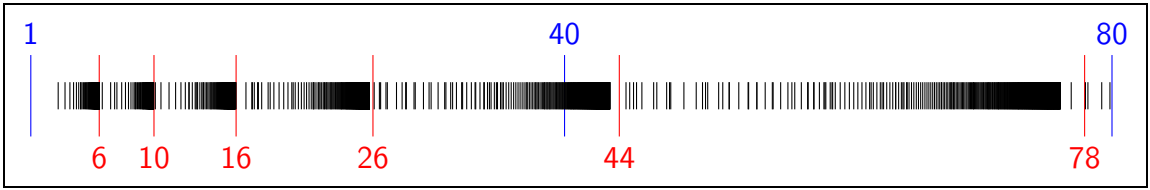
\begin{figure}[htb!]
 \centering
 \caption{Abscissae of convergence $\alpha(m,n)$ within the real
   interval $[0,80]$ arising from parameters $2 \le m \le 500$ and $2
   \le n \le 20$}
 \label{fig:abscissae}
 \[
 \boxed{\begin{tikzpicture}[font=\sffamily,scale=.18] \foreach \i in
     {1,40,80} \draw[blue](\i,-2) -- (\i,4) node[anchor=south] {\i};
     \foreach \i in {6,10,16,26,44,78} \draw[red](\i,4) -- (\i,-2)
     node[anchor=north] {\i}; \foreach \x in
     {3.000, 3.500, 3.857, 4.125, 4.333, 4.500, 4.636, 4.667, 4.750,
       4.846, 4.929, 5.000, 5.062, 5.118, 5.167, 5.210, 5.250, 5.286,
       5.318, 5.348, 5.375, 5.400, 5.423, 5.444, 5.464, 5.483, 5.500,
       5.516, 5.531, 5.545, 5.559, 5.571, 5.583, 5.595, 5.605, 5.615,
       5.625, 5.634, 5.643, 5.651, 5.659, 5.667, 5.674, 5.681, 5.688,
       5.694, 5.700, 5.706, 5.711, 5.717, 5.722, 5.727, 5.732, 5.737,
       5.741, 5.746, 5.750, 5.754, 5.758, 5.762, 5.766, 5.769, 5.773,
       5.776, 5.779, 5.783, 5.786, 5.789, 5.792, 5.794, 5.797, 5.800,
       5.803, 5.805, 5.808, 5.810, 5.812, 5.815, 5.817, 5.819, 5.821,
       5.824, 5.826, 5.828, 5.830, 5.832, 5.833, 5.835, 5.837, 5.839,
       5.840, 5.842, 5.844, 5.845, 5.847, 5.849, 5.850, 5.852, 5.853,
       5.854, 5.856, 5.857, 5.858, 5.860, 5.861, 5.862, 5.864, 5.865,
       5.866, 5.867, 5.869, 5.870, 5.871, 5.872, 5.873, 5.874, 5.875,
       5.876, 5.877, 5.878, 5.879, 5.880, 5.881, 5.882, 5.883, 5.884,
       5.885, 5.885, 5.886, 5.887, 5.888, 5.889, 5.890, 5.891, 5.891,
       5.892, 5.893, 5.894, 5.895, 5.895, 5.896, 5.896, 5.897, 5.898,
       5.898, 5.899, 5.900, 5.901, 5.901, 5.902, 5.903, 5.903, 5.904,
       5.904, 5.905, 5.906, 5.906, 5.907, 5.907, 5.908, 5.909, 5.909,
       5.910, 5.910, 5.911, 5.911, 5.912, 5.912, 5.913, 5.913, 5.914,
       5.914, 5.915, 5.915, 5.916, 5.916, 5.917, 5.917, 5.917, 5.918,
       5.918, 5.919, 5.919, 5.920, 5.920, 5.921, 5.921, 5.922, 5.922,
       5.923, 5.923, 5.924, 5.924, 5.925, 5.925, 5.926, 5.926, 5.927,
       5.927, 5.928, 5.928, 5.929, 5.929, 5.930, 5.930, 5.931, 5.931,
       5.932, 5.932, 5.933, 5.933, 5.934, 5.934, 5.935, 5.935, 5.936,
       5.936, 5.937, 5.937, 5.938, 5.938, 5.938, 5.939, 5.939, 5.940,
       5.940, 5.941, 5.941, 5.942, 5.942, 5.943, 5.943, 5.944, 5.944,
       5.945, 5.945, 5.946, 5.946, 5.947, 5.947, 5.948, 5.948, 5.949,
       5.949, 5.950, 5.950, 5.951, 5.951, 5.952, 5.952, 5.953, 5.953,
       5.954, 5.954, 5.955, 5.955, 5.956, 5.956, 5.957, 5.957, 5.958,
       5.958, 5.958, 5.959, 5.959, 5.960, 5.960, 5.961, 5.961, 5.962,
       5.962, 5.963, 5.963, 5.964, 5.964, 5.965, 5.965, 5.966, 5.966,
       5.967, 5.967, 5.968, 5.968, 5.969, 5.969, 5.970, 5.970, 6.222,
       6.833, 7.125, 7.291, 7.641, 7.917, 8.138, 8.318, 8.469, 8.596,
       8.643, 8.704, 8.797, 8.878, 8.948, 9.012, 9.067, 9.117, 9.162,
       9.203, 9.208, 9.240, 9.274, 9.305, 9.333, 9.359, 9.384, 9.406,
       9.428, 9.447, 9.465, 9.482, 9.498, 9.514, 9.527, 9.541, 9.554,
       9.565, 9.577, 9.588, 9.599, 9.608, 9.617, 9.626, 9.635, 9.643,
       9.650, 9.658, 9.665, 9.672, 9.679, 9.685, 9.690, 9.696, 9.702,
       9.708, 9.713, 9.718, 9.723, 9.728, 9.732, 9.736, 9.741, 9.745,
       9.749, 9.750, 9.753, 9.757, 9.760, 9.764, 9.768, 9.771, 9.773,
       9.776, 9.780, 9.783, 9.786, 9.789, 9.791, 9.794, 9.797, 9.799,
       9.802, 9.804, 9.807, 9.809, 9.811, 9.813, 9.815, 9.817, 9.819,
       9.821, 9.823, 9.825, 9.827, 9.829, 9.831, 9.833, 9.834, 9.836,
       9.838, 9.840, 9.841, 9.843, 9.844, 9.846, 9.847, 9.849, 9.850,
       9.852, 9.853, 9.854, 9.855, 9.856, 9.858, 9.859, 9.860, 9.861,
       9.863, 9.864, 9.865, 9.866, 9.867, 9.868, 9.869, 9.870, 9.872,
       9.873, 9.874, 9.875, 9.876, 9.877, 9.878, 9.879, 9.880, 9.881,
       9.882, 9.883, 9.884, 9.885, 9.886, 9.887, 9.888, 9.889, 9.890,
       9.891, 9.892, 9.893, 9.894, 9.895, 9.896, 9.896, 9.897, 9.898,
       9.899, 9.900, 9.901, 9.902, 9.903, 9.904, 9.905, 9.906, 9.907,
       9.908, 9.909, 9.910, 9.911, 9.912, 9.913, 9.914, 9.915, 9.916,
       9.917, 9.918, 9.919, 9.920, 9.921, 9.922, 9.923, 9.924, 9.925,
       9.926, 9.927, 9.928, 9.929, 9.930, 9.931, 9.932, 9.933, 9.934,
       9.935, 9.936, 9.937, 9.938, 9.938, 9.939, 9.940, 9.941, 9.942,
       9.943, 9.944, 9.945, 9.946, 9.947, 9.948, 9.949, 9.950, 9.951,
       9.952, 9.953, 9.954, 9.955, 9.956, 9.957, 9.958, 9.959, 9.960,
       9.961, 9.962, 9.963, 9.964, 9.965, 9.966, 9.967, 9.968, 10.10,
       10.51, 11.10, 11.58, 11.97, 12.29, 12.55, 12.57, 12.80, 13.01,
       13.08, 13.15, 13.19, 13.35, 13.49, 13.58, 13.62, 13.73, 13.83,
       13.93, 14.02, 14.10, 14.17, 14.17, 14.24, 14.30, 14.36, 14.41,
       14.46, 14.51, 14.56, 14.60, 14.64, 14.68, 14.71, 14.75, 14.78,
       14.81, 14.84, 14.87, 14.89, 14.92, 14.94, 14.96, 14.99, 15.01,
       15.03, 15.05, 15.07, 15.08, 15.10, 15.12, 15.13, 15.15, 15.16,
       15.17, 15.18, 15.19, 15.21, 15.22, 15.23, 15.24, 15.25, 15.27,
       15.28, 15.29, 15.30, 15.31, 15.32, 15.33, 15.34, 15.35, 15.35,
       15.36, 15.37, 15.38, 15.39, 15.40, 15.40, 15.41, 15.42, 15.42,
       15.43, 15.44, 15.44, 15.45, 15.46, 15.46, 15.47, 15.47, 15.48,
       15.49, 15.49, 15.50, 15.50, 15.51, 15.51, 15.52, 15.52, 15.53,
       15.53, 15.54, 15.54, 15.54, 15.55, 15.55, 15.56, 15.56, 15.57,
       15.57, 15.57, 15.58, 15.58, 15.58, 15.59, 15.59, 15.59, 15.60,
       15.60, 15.60, 15.61, 15.61, 15.61, 15.62, 15.62, 15.62, 15.63,
       15.63, 15.63, 15.63, 15.64, 15.64, 15.64, 15.65, 15.65, 15.65,
       15.65, 15.66, 15.66, 15.66, 15.66, 15.67, 15.67, 15.67, 15.67,
       15.67, 15.68, 15.68, 15.68, 15.68, 15.69, 15.69, 15.69, 15.69,
       15.69, 15.70, 15.70, 15.70, 15.70, 15.70, 15.70, 15.71, 15.71,
       15.71, 15.71, 15.71, 15.71, 15.72, 15.72, 15.72, 15.72, 15.72,
       15.72, 15.73, 15.73, 15.73, 15.73, 15.73, 15.73, 15.74, 15.74,
       15.74, 15.74, 15.74, 15.74, 15.74, 15.75, 15.75, 15.75, 15.75,
       15.75, 15.75, 15.75, 15.75, 15.76, 15.76, 15.76, 15.76, 15.76,
       15.76, 15.76, 15.76, 15.77, 15.77, 15.77, 15.77, 15.77, 15.77,
       15.77, 15.77, 15.77, 15.78, 15.78, 15.78, 15.78, 15.78, 15.78,
       15.78, 15.78, 15.78, 15.78, 15.79, 15.79, 15.79, 15.79, 15.79,
       15.79, 15.79, 15.79, 15.79, 15.79, 15.79, 15.80, 15.80, 15.80,
       15.80, 15.80, 15.80, 15.80, 15.80, 15.80, 15.80, 15.81, 15.81,
       15.81, 15.81, 15.81, 15.81, 15.81, 15.81, 15.81, 15.81, 15.82,
       15.82, 15.82, 15.82, 15.82, 15.82, 15.82, 15.82, 15.82, 15.82,
       15.83, 15.83, 15.83, 15.83, 15.83, 15.83, 15.83, 15.83, 15.83,
       15.83, 15.83, 15.84, 15.84, 15.84, 15.84, 15.84, 15.84, 15.84,
       15.84, 15.84, 15.84, 15.85, 15.85, 15.85, 15.85, 15.85, 15.85,
       15.85, 15.85, 15.85, 15.85, 15.86, 15.86, 15.86, 15.86, 15.86,
       15.86, 15.86, 15.86, 15.86, 15.86, 15.87, 15.87, 15.87, 15.87,
       15.87, 15.87, 15.87, 15.87, 15.87, 15.87, 15.88, 15.88, 15.88,
       15.88, 15.88, 15.88, 15.88, 15.88, 15.88, 15.88, 15.88, 15.89,
       15.89, 15.89, 15.89, 15.89, 15.89, 15.89, 15.89, 15.89, 15.89,
       15.90, 15.90, 15.90, 15.90, 15.90, 15.90, 15.90, 15.90, 15.90,
       15.90, 16.00, 16.72, 17.15, 17.19, 17.33, 17.57, 17.63, 17.87,
       18.34, 18.49, 18.76, 19.14, 19.48, 19.79, 20.07, 20.17, 20.32,
       20.55, 20.77, 20.97, 21.15, 21.33, 21.49, 21.63, 21.64, 21.68,
       21.77, 21.91, 22.03, 22.06, 22.14, 22.25, 22.36, 22.43, 22.46,
       22.55, 22.63, 22.63, 22.72, 22.80, 22.88, 22.94, 22.95, 23.02,
       23.08, 23.14, 23.21, 23.26, 23.32, 23.37, 23.42, 23.47, 23.52,
       23.57, 23.61, 23.65, 23.69, 23.71, 23.73, 23.77, 23.81, 23.84,
       23.88, 23.91, 23.95, 23.98, 24.01, 24.04, 24.06, 24.09, 24.11,
       24.12, 24.15, 24.17, 24.20, 24.22, 24.24, 24.27, 24.29, 24.31,
       24.33, 24.35, 24.38, 24.39, 24.41, 24.43, 24.45, 24.47, 24.49,
       24.50, 24.52, 24.54, 24.55, 24.57, 24.58, 24.60, 24.61, 24.63,
       24.64, 24.66, 24.67, 24.68, 24.70, 24.71, 24.72, 24.73, 24.75,
       24.76, 24.77, 24.78, 24.79, 24.80, 24.81, 24.82, 24.83, 24.84,
       24.85, 24.86, 24.87, 24.88, 24.89, 24.90, 24.91, 24.92, 24.93,
       24.94, 24.95, 24.95, 24.96, 24.97, 24.98, 24.99, 24.99, 25.00,
       25.01, 25.02, 25.02, 25.03, 25.04, 25.04, 25.05, 25.06, 25.06,
       25.07, 25.08, 25.08, 25.09, 25.10, 25.10, 25.11, 25.12, 25.12,
       25.13, 25.13, 25.14, 25.14, 25.15, 25.15, 25.16, 25.17, 25.17,
       25.18, 25.18, 25.19, 25.19, 25.20, 25.20, 25.21, 25.21, 25.21,
       25.22, 25.22, 25.23, 25.23, 25.24, 25.24, 25.25, 25.25, 25.26,
       25.26, 25.26, 25.27, 25.27, 25.28, 25.28, 25.28, 25.29, 25.29,
       25.29, 25.30, 25.30, 25.31, 25.31, 25.31, 25.32, 25.32, 25.32,
       25.33, 25.33, 25.33, 25.34, 25.34, 25.34, 25.35, 25.35, 25.35,
       25.36, 25.36, 25.36, 25.37, 25.37, 25.37, 25.38, 25.38, 25.38,
       25.38, 25.39, 25.39, 25.39, 25.39, 25.40, 25.40, 25.40, 25.41,
       25.41, 25.41, 25.41, 25.42, 25.42, 25.42, 25.42, 25.43, 25.43,
       25.43, 25.43, 25.44, 25.44, 25.44, 25.44, 25.45, 25.45, 25.45,
       25.45, 25.46, 25.46, 25.46, 25.46, 25.46, 25.47, 25.47, 25.47,
       25.47, 25.47, 25.48, 25.48, 25.48, 25.48, 25.48, 25.49, 25.49,
       25.49, 25.49, 25.49, 25.50, 25.50, 25.50, 25.50, 25.50, 25.51,
       25.51, 25.51, 25.51, 25.51, 25.52, 25.52, 25.52, 25.52, 25.52,
       25.53, 25.53, 25.53, 25.53, 25.53, 25.54, 25.54, 25.54, 25.54,
       25.54, 25.54, 25.55, 25.55, 25.55, 25.55, 25.55, 25.56, 25.56,
       25.56, 25.56, 25.56, 25.57, 25.57, 25.57, 25.57, 25.57, 25.58,
       25.58, 25.58, 25.58, 25.58, 25.59, 25.59, 25.59, 25.59, 25.59,
       25.60, 25.60, 25.60, 25.60, 25.60, 25.61, 25.61, 25.61, 25.61,
       25.61, 25.62, 25.62, 25.62, 25.62, 25.62, 25.62, 25.63, 25.63,
       25.63, 25.63, 25.63, 25.64, 25.64, 25.64, 25.64, 25.64, 25.65,
       25.65, 25.65, 25.65, 25.65, 25.66, 25.66, 25.66, 25.66, 25.66,
       25.67, 25.67, 25.67, 25.67, 25.67, 25.68, 25.68, 25.68, 25.68,
       25.68, 25.69, 25.69, 25.69, 25.69, 25.69, 25.70, 25.70, 25.70,
       25.70, 25.70, 25.71, 25.71, 25.71, 25.71, 25.71, 25.71, 25.72,
       25.72, 25.72, 25.72, 25.72, 25.73, 25.73, 25.73, 25.73, 25.73,
       25.74, 25.74, 26.09, 26.47, 26.56, 26.95, 27.05, 27.72, 28.13,
       28.39, 28.44, 29.00, 29.09, 29.69, 30.10, 30.25, 30.76, 31.24,
       31.33, 31.68, 31.82, 32.10, 32.48, 32.55, 32.85, 33.19, 33.46,
       33.51, 33.82, 34.10, 34.14, 34.37, 34.51, 34.63, 34.87, 35.02,
       35.10, 35.32, 35.43, 35.53, 35.73, 35.92, 36.10, 36.28, 36.44,
       36.60, 36.75, 36.90, 37.04, 37.18, 37.24, 37.31, 37.43, 37.55,
       37.67, 37.78, 37.89, 37.91, 38.00, 38.09, 38.20, 38.29, 38.38,
       38.47, 38.55, 38.55, 38.64, 38.72, 38.72, 38.80, 38.87, 38.92,
       38.95, 39.02, 39.09, 39.16, 39.22, 39.29, 39.35, 39.41, 39.47,
       39.53, 39.58, 39.64, 39.69, 39.75, 39.80, 39.84, 39.89, 39.94,
       39.99, 40.04, 40.08, 40.12, 40.16, 40.21, 40.25, 40.29, 40.32,
       40.36, 40.40, 40.44, 40.47, 40.47, 40.51, 40.54, 40.58, 40.61,
       40.64, 40.66, 40.68, 40.71, 40.73, 40.77, 40.80, 40.82, 40.85,
       40.88, 40.91, 40.94, 40.96, 40.99, 41.02, 41.04, 41.06, 41.09,
       41.11, 41.14, 41.16, 41.18, 41.20, 41.23, 41.25, 41.27, 41.29,
       41.31, 41.33, 41.35, 41.37, 41.39, 41.41, 41.43, 41.45, 41.47,
       41.48, 41.50, 41.52, 41.54, 41.55, 41.57, 41.59, 41.61, 41.62,
       41.64, 41.66, 41.67, 41.68, 41.70, 41.71, 41.73, 41.75, 41.76,
       41.77, 41.79, 41.80, 41.82, 41.83, 41.84, 41.86, 41.87, 41.88,
       41.89, 41.91, 41.92, 41.93, 41.95, 41.96, 41.97, 41.98, 41.99,
       42.00, 42.02, 42.02, 42.04, 42.05, 42.06, 42.07, 42.08, 42.09,
       42.10, 42.11, 42.12, 42.13, 42.14, 42.15, 42.16, 42.17, 42.18,
       42.19, 42.20, 42.21, 42.21, 42.23, 42.23, 42.24, 42.25, 42.26,
       42.27, 42.28, 42.29, 42.29, 42.30, 42.31, 42.32, 42.32, 42.33,
       42.34, 42.34, 42.35, 42.36, 42.37, 42.38, 42.38, 42.39, 42.39,
       42.40, 42.41, 42.42, 42.43, 42.43, 42.44, 42.45, 42.45, 42.46,
       42.46, 42.47, 42.48, 42.48, 42.49, 42.50, 42.50, 42.51, 42.52,
       42.52, 42.53, 42.54, 42.54, 42.55, 42.55, 42.56, 42.57, 42.57,
       42.58, 42.59, 42.59, 42.59, 42.60, 42.61, 42.61, 42.62, 42.62,
       42.63, 42.63, 42.64, 42.64, 42.65, 42.66, 42.66, 42.66, 42.67,
       42.68, 42.68, 42.68, 42.69, 42.70, 42.70, 42.70, 42.71, 42.71,
       42.72, 42.72, 42.73, 42.73, 42.74, 42.74, 42.75, 42.75, 42.75,
       42.76, 42.77, 42.77, 42.77, 42.78, 42.78, 42.79, 42.79, 42.79,
       42.80, 42.80, 42.80, 42.81, 42.81, 42.82, 42.82, 42.83, 42.83,
       42.84, 42.84, 42.84, 42.85, 42.85, 42.86, 42.86, 42.86, 42.87,
       42.87, 42.88, 42.88, 42.88, 42.89, 42.89, 42.89, 42.90, 42.90,
       42.91, 42.91, 42.91, 42.92, 42.92, 42.93, 42.93, 42.93, 42.94,
       42.94, 42.95, 42.95, 42.95, 42.96, 42.96, 42.96, 42.97, 42.97,
       42.98, 42.98, 42.98, 42.99, 42.99, 43.00, 43.00, 43.00, 43.01,
       43.01, 43.02, 43.02, 43.02, 43.03, 43.03, 43.04, 43.04, 43.04,
       43.05, 43.05, 43.05, 43.06, 43.06, 43.07, 43.07, 43.07, 43.08,
       43.08, 43.09, 43.09, 43.09, 43.10, 43.10, 43.11, 43.11, 43.11,
       43.12, 43.12, 43.12, 43.13, 43.13, 43.14, 43.14, 43.14, 43.15,
       43.15, 43.16, 43.16, 43.16, 43.17, 43.17, 43.18, 43.18, 43.18,
       43.19, 43.19, 43.20, 43.20, 43.20, 43.21, 43.21, 43.21, 43.22,
       43.22, 43.23, 43.23, 43.23, 43.24, 43.24, 43.25, 43.25, 43.25,
       43.26, 43.26, 43.27, 43.27, 43.27, 43.28, 43.28, 43.29, 43.29,
       43.29, 43.30, 43.30, 43.30, 43.31, 44.49, 44.76, 45.04, 45.26,
       45.65, 46.51, 46.74, 47.46, 47.67, 47.76, 48.72, 49.62, 50.08,
       50.30, 50.46, 51.26, 51.56, 52.02, 52.04, 52.73, 53.41, 53.45,
       54.05, 54.05, 54.66, 55.18, 55.24, 55.79, 56.32, 56.61, 56.82,
       57.30, 57.76, 57.80, 58.20, 58.23, 58.62, 58.95, 59.02, 59.09,
       59.41, 59.54, 59.59, 59.78, 60.14, 60.48, 60.81, 61.13, 61.43,
       61.73, 62.02, 62.29, 62.40, 62.56, 62.82, 63.07, 63.20, 63.31,
       63.54, 63.77, 63.98, 64.19, 64.20, 64.41, 64.61, 64.72, 64.80,
       64.99, 65.04, 65.17, 65.35, 65.52, 65.70, 65.86, 66.02, 66.17,
       66.33, 66.48, 66.62, 66.76, 66.90, 67.03, 67.16, 67.29, 67.41,
       67.53, 67.54, 67.66, 67.77, 67.89, 68.00, 68.11, 68.21, 68.26,
       68.32, 68.42, 68.46, 68.52, 68.62, 68.72, 68.81, 68.90, 68.99,
       69.08, 69.16, 69.25, 69.34, 69.41, 69.50, 69.58, 69.66, 69.73,
       69.80, 69.88, 69.88, 69.95, 70.02, 70.09, 70.16, 70.23, 70.29,
       70.36, 70.42, 70.48, 70.55, 70.61, 70.66, 70.73, 70.78, 70.84,
       70.90, 70.95, 70.99, 71.01, 71.06, 71.11, 71.16, 71.22, 71.27,
       71.31, 71.37, 71.41, 71.46, 71.51, 71.55, 71.60, 71.64, 71.69,
       71.71, 71.73, 71.77, 71.82, 71.86, 71.90, 71.94, 71.98, 72.02,
       72.06, 72.10, 72.14, 72.17, 72.21, 72.25, 72.28, 72.32, 72.36,
       72.39, 72.42, 72.46, 72.49, 72.52, 72.56, 72.59, 72.62, 72.66,
       72.69, 72.72, 72.75, 72.78, 72.80, 72.84, 72.87, 72.90, 72.92,
       72.95, 72.98, 73.01, 73.04, 73.06, 73.09, 73.12, 73.14, 73.16,
       73.17, 73.20, 73.22, 73.24, 73.27, 73.30, 73.32, 73.34, 73.37,
       73.39, 73.41, 73.44, 73.46, 73.48, 73.50, 73.52, 73.55, 73.57,
       73.59, 73.61, 73.63, 73.66, 73.67, 73.70, 73.72, 73.73, 73.76,
       73.77, 73.80, 73.81, 73.84, 73.85, 73.88, 73.89, 73.91, 73.93,
       73.95, 73.96, 73.98, 74.00, 74.02, 74.03, 74.05, 74.07, 74.09,
       74.10, 74.12, 74.13, 74.16, 74.17, 74.19, 74.20, 74.20, 74.22,
       74.23, 74.25, 74.27, 74.28, 74.30, 74.31, 74.33, 74.34, 74.35,
       74.37, 74.38, 74.40, 74.41, 74.43, 74.44, 74.45, 74.47, 74.48,
       74.49, 74.51, 74.52, 74.54, 74.55, 74.56, 74.58, 74.59, 74.60,
       74.61, 74.62, 74.64, 74.65, 74.66, 74.67, 74.69, 74.70, 74.71,
       74.73, 74.73, 74.75, 74.76, 74.77, 74.78, 74.79, 74.80, 74.81,
       74.83, 74.84, 74.85, 74.86, 74.87, 74.88, 74.89, 74.91, 74.91,
       74.92, 74.94, 74.95, 74.95, 74.96, 74.98, 74.98, 74.99, 75.01,
       75.02, 75.02, 75.03, 75.05, 75.05, 75.06, 75.07, 75.08, 75.09,
       75.10, 75.11, 75.12, 75.12, 75.14, 75.15, 75.16, 75.16, 75.17,
       75.18, 75.19, 75.20, 75.21, 75.22, 75.23, 75.23, 75.24, 75.25,
       75.26, 75.27, 75.27, 75.28, 75.29, 75.30, 75.30, 75.31, 75.32,
       75.33, 75.34, 75.34, 75.35, 75.36, 75.37, 75.38, 75.38, 75.39,
       75.40, 75.41, 75.41, 75.42, 75.43, 75.44, 75.45, 75.45, 75.46,
       75.47, 75.48, 75.48, 75.49, 75.50, 75.51, 75.52, 75.52, 75.53,
       75.54, 75.55, 75.55, 75.56, 75.57, 75.58, 75.59, 75.59, 75.60,
       75.61, 75.62, 75.62, 75.63, 75.64, 75.65, 75.66, 75.66, 75.67,
       75.68, 75.69, 75.70, 75.70, 75.71, 75.72, 75.73, 75.73, 75.74,
       75.75, 75.76, 75.77, 75.77, 75.78, 75.79, 75.80, 75.80, 75.81,
       75.82, 75.83, 75.84, 75.84, 75.85, 75.86, 75.87, 75.88, 75.88,
       75.89, 75.90, 75.91, 75.91, 75.92, 75.93, 75.94, 75.95, 75.95,
       75.96, 75.97, 75.98, 75.98, 75.99, 76.00, 76.01, 76.02, 76.02,
       76.03, 76.04, 76.05, 76.05, 76.06, 76.07, 76.08, 76.09, 76.09,
       76.10, 76.11, 76.12, 76.12, 76.13, 76.14, 76.15, 76.16, 76.16,
       76.17, 76.18, 76.19, 76.20, 76.20, 76.21, 76.22, 77.02, 78.06,
       78.24, 79.25, 79.85}
     \draw (\x,0) -- (\x,2);
   \end{tikzpicture}}
 \]
 The first six cluster points provided by
 Corollary~\ref{cor:limit-points}, the numbers $6,10,16,26,44,78$ are
 indicated in red.  One can observe that the rate of convergence is
 fairly slow.
\end{figure}

\subsection{Methods} \label{subsec:Lie-corr} Finally we give a brief
indication how pro-isomorphic zeta functions of nilpotent groups
relate to zeta functions of algebraic groups.  This gives us the
opportunity to discuss the methods used and the relevance of our
specific results in the context of the general theory.

To a finitely generated, torsion-free nilpotent group $\Gamma$ one
associates, via Lie theory, a $\Z$-Lie lattice $L$ of finite rank,
whose local zeta functions $\zeta^\wedge_{L_p}(s) = \sum_{k=0}^\infty
b_{p^k}^\wedge(L_p)p^{-ks}$ satisfy $\zeta^\wedge_{\Gamma,p}(s) =
\zeta^\wedge_{L_p}(s)$ for almost all rational primes~$p$.  Here $L_p
= \Z_p \otimes_{\Z} L$ denotes the $p$-adic completion of $L$, and
$b_{p^k}^\wedge(L_p)$ is the number of Lie sublattices of $L_p$ of
index $p^k$ which are isomorphic to $L_p$,  or equivalently, are the
image of $L_p$ under a Lie endomorphism of $\Q_p \otimes L$.  Next
recall the notion of the \emph{algebraic automorphism group} of~$L$:
this group $\aG = \mathsf{Aut}(L)$ is realised, via a $\Z$-basis of
$L$, as an affine group scheme $\aG \subseteq \GL_d$ over
$\mathbb{Z}$, where $d$ denotes the $\Z$-rank of $L$, so that
$\Aut_K(K \otimes_\Z L) \cong \aG(K) \subseteq \GL_{d}(K)$ for every
extension field $K$ of~$\Q$.  With the given arithmetic structure, we
have $\Aut(L) \cong \aG(\Z)$ and $\Aut(L_p) \cong \aG(\Z_p)$ for every
rational prime~$p$.  In~~\cite{GrSeSm88}, Grunewald, Segal and Smith
showed that
\begin{equation*} \zeta^\wedge_{\Gamma,p}(s) = \zeta_{L_p}^{\wedge}(s)
  = \mathcal{Z}_{\aG,p}(s), 
\end{equation*}
where $\mathcal{Z}_{\aG,p}(s)$ is a special instance of the
$\fp$-adic integral~\eqref{equ:zeta-alg-group}; the first equality
holds for almost all primes and the second for all primes~$p$.

An explicit finite form for $\mathcal{Z}_{\aG,\fp}(s)$, subject to
technical restrictions on $\aG \subseteq \GL_d$, was obtained by
Igusa~\cite{Ig89}, subsequently generalised by du Sautoy and
Lubotzky~\cite{dSLu96} and by the first author~\cite{Be11}.  The
essential idea behind these results is to reduce the
integral~\eqref{equ:zeta-alg-group} to an integral over a reductive
$\fp$-adic group $\aH_\fp$ and then apply a $\fp$-adic Bruhat
decomposition established by Iwahori and Matsumoto~\cite{IwMa65}.
Upon further combinatorial analysis one obtains a weighted sum of
generating functions over cones, indexed by elements of the Weyl group
of~$\aH$.

The reduction of the integral~\eqref{equ:zeta-alg-group} to an
integral over a connected reductive group, due to du Sautoy and
Lubotzky, can be summarised as follows.

\begin{theorem}[{cf.~\cite[Theorem~2.2]{dSLu96}}]\label{theorem:reduction}
  Under various assumptions on $\aG$ and~$\fp$, which are satisfied
  for almost all primes~$\fp$,
  \[
  \mathcal{Z}_{\aG,\fp}(s) = \int_{\aH^+_\fp} \lvert \det(h)
  \rvert_\fp^s \, \theta(h) \, d\mu_{\aH_\fp}(h),
  \]
  where $\aH$ is a connected reductive complement in $\aG^\circ$ of
  the unipotent radical $\aN$ of $\aG$ and $\theta(h) =
  \mu_{\aN_\fp}(\{n \in \aN_\fp \mid nh \in \aG_\fp^+ \})$ for each $h
  \in \aH_\fp^+$.
\end{theorem}

Despite an encouraging range of valuable insights, our understanding
of pro-isomorphic zeta functions of finitely generated nilpotent
groups is far from complete.  For instance, in \cite{Be11}, the first
author generalised earlier results of Igusa~\cite{Ig89} and du Sautoy
and Lubotzky~\cite{dSLu96} showing that the pro-isomorphic zeta
functions of split algebraic groups satisfy -- under suitable
technical conditions -- local functional equations.  However, it is
currently not possible to effectively predict whether the technical
assumptions involved hold for the algebraic automorphism group
associated to a finitely generated nilpotent group without actually
determining the automorphism groups in full.  Moreover, the algebraic
automorphism groups that have been described explicitly display only a
comparatively small degree of complexity, especially regarding the
unipotent radical.

It is a natural task to reveal more variety in the features of
pro-isomorphic zeta functions by studying new families of groups where
we continue to have some level of control.  In the current paper we
treat the finitely generated, torsion-free class-$2$ nilpotent groups
$\Delta_{m,n}$ and obtain further evidence for the existence of local
functional equations for pro-isomorphic zeta functions in nilpotency
class~$2$.

The function $\theta(h)$ that features in
Theorem~\ref{theorem:reduction} is perhaps the least understood
ingredient in the study of pro-isomorphic zeta functions of nilpotent
groups.  In~\cite[Theorem~2.3]{dSLu96}, du Sautoy and Lubotzky state
that, if the $\fp$-adic group $\aH_\fp$ arises from the reductive part
$\aH$ of the algebraic automorphism group associated to an arbitrary
class-$2$ nilpotent group, then the function $\theta(h)$ is a
character on~$\aH_\fp$.  In the current paper, we find counterexamples
to this statement. Indeed, for the algebraic automorphism groups
associated to $\Delta_{m,n}$ we find that $\theta(h)$ is obtained in
general not from a character but rather from a piecewise-character on
a maximal torus in the sense considered in \cite[Lemma~3.12]{Be11},
thus giving a realisation of this behaviour of $\theta(h)$ for an
infinite family of groups.  Previously this phenomenon was known to
occur only from one isolated example, namely the group of upper
unitriangular $4 \times 4$ matrices over~$\Z$.

In our preprint~\cite{BeKlOnX0}, on pro-isomorphic zeta functions of
$D^*$-groups of even Hirsch length, we describe an example of a
$D^*$-group for which the function $\theta(h)$ is highly exotic and
very far from being a character on~$\aH_\fp$.  Thus even within the
comparatively restrictive setting of class-$2$ nilpotent groups, the
picture is much more complex than previously expected and deserves
further study.

\begin{ackno}
  The first author thanks the University of Cape Town and Ort Braude
  College for travel grants.  We are grateful to Christopher Voll for
  several helpful comments on an early draft of the paper.
\end{ackno}


\section{The Lie lattices $L_{m,n}$ and their algebraic automorphism groups}

\begin{definition} \label{def-Lmn}

Let $m,n \in \N$ with $n \geq 2$.  Put
\begin{align*}
  \mathbf{E} & = \{ \ee \mid \ee = (e_1, \ldots, e_n) \in \N_0^n
  \text{ with $e_1 + \ldots + e_n = m-1$} \}, \\
  \mathbf{F} & = \{ \ff \mid \ff = (f_1,\ldots,f_n) \in \N_0^n \text{
    with $f_1 + \ldots + f_n = m$} \}.
\end{align*}

We consider the $\Z$-Lie lattice $L = L_{m,n}$ on the generators
\begin{equation} \label{equ:xyz-basis} \{ x_\ee, \mid \ee \in
  \mathbf{E} \} \cup \{ y_\ff \mid \ff \in \mathbf{F} \} \cup \{ z_j
  \mid j \in \{1, \ldots, n\} \},
\end{equation}
so that $L$ has $\Z$-rank
\[
\mathrm{rank}_\Z( L) = \binom{m+n-2}{n-1} + \binom{m+n-1}{n-1} + n,
\]
subject to the defining relations
\[
[x_\ee,x_{\ee'}] = [y_\ff,y_{\ff'}] = [x_\ee,z_j] = [y_\ff,z_j] =
[z_j,z_{j'}] = 0
\]
for all $\ee, \ee' \in \mathbf{E}$, $\, \ff, \ff' \in \mathbf{F}$, $\, j, j'
\in \{1,\ldots,n\}$ and
\[
[x_\ee,y_\ff] =
\begin{cases} 
  z_i & \text{if $\ff - \ee$ is of the form
    $(\underbrace{0,\ldots,0}_{i-1 \text{
        entries}},1,\underbrace{0,\ldots,0}_{n-i \text{ entries}})$
    for some $i$,} \\
  0 & \text{if $\ff - \ee$ is not of this form}
\end{cases}
\]
for $\ee \in \mathbf{E}$ and $\ff \in \mathbf{F}$.  Consequently, $L$
is nilpotent of class $2$, and
\[
\mathrm{Z}(L) = [L,L] = \langle z_1, \ldots, z_n \rangle
\]
has $\Z$-rank~$n$.  
\end{definition}

\begin{remark}\label{remark:Lie-correspondence}
  We observe that $L = L_{m,n}$ can be naturally identified with a
  graded Lie ring associated to the class-$2$ nilpotent group
  $\Delta = \Delta_{m,n}$ introduced in
  Section~\ref{sec:groups-Delta}.  Indeed, because $\Delta$ has class
  $2$, it is elementary to set up a naive Lie correspondence as
  follows: $\Delta / \mathrm{Z}(\Delta) \oplus \mathrm{Z}(\Delta)$
  carries the structure of a graded Lie ring with respect to the usual
  Lie bracket operation induced by commutation in~$\Delta$.
  Conversely, the Lie ring $L$ fully determines the group~$\Delta$.
  For any rational prime $p$, we denote by $L_p = \Z_p \otimes_{\Z} L$
  the $p$-adic completion of $L$ and by $b_{p^k}^\wedge(L_p)$ the
  number of Lie sublattices of $L_p$ of index $p^k$ that are
  isomorphic to $L_p$.  Then the local zeta function
  $\zeta^\wedge_{L_p}(s) = \sum_{k=0}^\infty
  b_{p^k}^\wedge(L_p)p^{-ks}$
  coincides with the local zeta function $\zeta^\wedge_{\Delta,p}(s)$
  of the group~$\Delta$.
\end{remark}

We fix a field $K$ and put $L_K = K \otimes_\Z L$.  Our task is to
determine the automorphism group $\Aut(L_K)$ as a subgroup
of~$\GL(L_K)$, with reference to the basis~\eqref{equ:xyz-basis}.

\begin{lemma} \label{lem:max-ab-ideal} The $K$-space $A = \langle
  y_\ff \mid \ff \in \mathbf{F} \rangle + \mathrm{Z}(L_K)$ constitutes the unique
  abelian Lie-ideal of dimension $\binom{m+n-1}{n-1} + n$ in~$L_K$.
  Consequently, $A$ is $\Aut(L_K)$-invariant.
\end{lemma}

\begin{proof}
  Clearly, $A$ is an abelian ideal of dimension
  $\binom{m+n-1}{n-1} + n$ in~$L_K$.  Suppose that $B$ is a maximal
  abelian Lie-ideal of $L_K$ and $B \not \subseteq A$.  It suffices to
  show that $\dim_K(B) < \dim_K(A)$.  Clearly,
  $\mathrm{Z}(L_K) \subseteq B$.  We work in
  $\overline{L_K} = L_K/ \mathrm{Z}(L_K) = \overline{V} \oplus
  \overline{W}$,
  where
  $\overline{V} = \langle \overline{x_\ee} \mid \ee \in \mathbf{E}
  \rangle$
  and
  $\overline{W} = \langle \overline{y_\ff} \mid \ff \in \mathbf{F}
  \rangle$.
  As $B \not \subseteq A$, we have $\overline{V_1} \not = 0$, where
  $\overline{V_1} \leq \overline{V}$ is such that
  $\overline{V_1} \oplus \overline{W} = \overline{B} + \overline{W}$.

  We construct a $K$-basis
  $\mathcal{B} = (\overline{v_1(1)}, \ldots, \overline{v_1(r)})$ of
  $\overline{V_1}$ as follows.  Fix $i \in \{1,\ldots,n\}$.  We define
  a lexicographical order on the set $\mathbf{E}$ by declaring, for
  $\ee, \ee' \in \mathbf{E}$,
  \begin{multline*}
    \ee \succ_i \ee' \qquad \text{if $e_i = e_i'$, $e_{i+1} =
    e_{i+1}'$, \ldots, $e_j = e_j'$ and  $e_{j+1} > e_{j+1}'$} \\
    \text{for some $j \in \{i-1,i,\ldots,n,1,2, \ldots, i-2\}$},
  \end{multline*}
  where -- here and in the following -- indices for the coordinates of
  $\ee$ and $\ee'$ are read in a circular way modulo~$n$.
  Accordingly, we write
  \[
  \ee \succeq_i \ee' \qquad \text{if } \ee = \ee' \text{ or } \ee
  \succ_i \ee'.
  \]
  Of course, one can define analogously an order on $\mathbf{F}$ which
  we shall simply refer to by the same symbols $\succ_i$ and
  $\succeq_i$.  By $\mathrm{lt}_i(\overline{v})$ we denote the leading
  term of $\overline{v} \in \overline{V}$ with respect to the ordered
  basis $(\overline{x_\ee} \mid \ee \in \mathbf{E}; \succeq_i)$.

  We now choose the basis $\mathcal{B} = (\overline{v_1(1)}, \ldots,
  \overline{v_1(r)})$ of $\overline{V_1}$ in such a way that
  \[
  \mathrm{lt}_i(\overline{v_1(j)}) = \overline{x_{\ee(j)}} \quad
  \text{for $1 \leq j \leq r$,} 
  \]
  where $\ee(1) \succ_i \ee(2) \succ_i \ldots \succ_i \ee(r)$ and the
  coefficient matrix of the vector system $(\overline{v_1(1)}, \ldots,
  \overline{v_1(r)})$ with respect to the basis $( \overline{x_\ee}
  \mid \ee \in \mathbf{E}; \succeq_i)$ of $\overline{V}$ has reduced
  echelon shape
  \[
  \begin{pmatrix}
    0 \, \cdots \, 0 & 1 & * \, \cdots \, * & 0 & * \, \cdots \, * & 0
    & * \, \cdots \\
    0 \, \cdots \phantom{\, 0} & & \phantom{\, 0} \cdots \, 0 & 1 & *
    \, \cdots \, * & 0 & * \, \cdots \\
    0 \, \cdots \phantom{\, 0} & & \cdots & & \phantom{\, 0} \cdots \, 0 & 1
    & * \, \cdots \\
    \cdots & & \cdots & & \cdots
  \end{pmatrix}.
  \]

  Next we consider $\overline{V_1}^\perp \leq \overline{W}$, where
  \[
  \overline{V_1}^\perp = \{ \overline{w} \in \overline{W} \mid \forall
  j \in \{1\ldots,r\} : [ \overline{v_1(j)}, \overline{w} ] = 0 \}
  \]
  is defined in terms of the induced `Lie bracket map'
  $L_K/ \mathrm{Z}(L_K) \times L_K/ \mathrm{Z}(L_K) \rightarrow
  Z(L_K)$.

  Below we show that
  \begin{equation}\label{equ:dim-formula-ab-ideal}
    \dim_K(\overline{V_1}^\perp) < \dim_K(\overline{W}) -
    \dim_K(\overline{V_1}).
  \end{equation}
  We observe that
  $\dim_K(\overline{B}) = \dim_K(\overline{V_1}) + \dim_K(\overline{B}
  \cap \overline{W})$
  and $\overline{B} \cap \overline{W} \subseteq \overline{V_1}^\perp$,
  the latter because $B$ is an abelian Lie-ideal.  This implies
  \[
  \dim_K(\overline{B}) \leq \dim_K(\overline{V_1}) +
  \dim_K(\overline{V_1}^\perp) < \dim_K(\overline{W}),
  \]
  thus $\dim_K(B) < \dim_K(A)$, as desired.

  \smallskip

  It remains to justify \eqref{equ:dim-formula-ab-ideal}.  For this it
  suffices to produce $\overline{w_1}, \ldots, \overline{w_{r+1}} \in
  \overline{W}$ such that
  \[
  \dim_K(\langle \overline{w_1}, \ldots, \overline{w_{r+1}} \rangle) =
  r+1 = \dim_K(\overline{V_1}) + 1 \quad \text{and} \quad \langle
  \overline{w_1}, \ldots, \overline{w_{r+1}} \rangle \cap
  \overline{V_1}^\perp = 0.
  \]
  Take $\overline{w_j} = \overline{y_{\ff(j)}}$ for $1 \leq j \leq
  r$, where 
  \[
  \ff(j) = \ee(j) + (\underbrace{0,\ldots,0}_{i-1 \text{
      entries}},1,\underbrace{0,\ldots,0}_{n-i \text{ entries}})
  \]
  and $\overline{w_{r+1}} = \overline{y_{\ff(r+1)}}$,
  where
  \[
  \ff(r+1) = \widetilde{\ee} + (\underbrace{0,\ldots,0}_{i-2 \text{
      entries}},1,\underbrace{0,\ldots,0}_{n-i+1 \text{
      entries}})
  \]
  and $\widetilde{\ee}$ is the $\succeq_i$-smallest  index such that at
  least one of $\overline{v_1(1)}$, \ldots, $\overline{v_1(r)}$ has a
  non-zero coefficient for $\overline{x_{\widetilde{\ee}}}$ with
  respect to the basis $(\overline{x_\ee} \mid \ee \in \mathbf{E};
  \succeq_i)$ of $\overline{V}$.

  To see that $\overline{w_1}, \ldots, \overline{w_{r+1}}$ are
  linearly independent, it suffices to show that the indices $\ff(1),
  \ldots, \ff(r+1)$ are pairwise distinct.  Clearly, this is true for
  $\ff(1), \ldots, \ff(r)$, because $\ee(1), \ldots, \ee(r)$ are
  pairwise distinct by construction.  It remains to show that $\ff(j)
  \not = \ff(r+1)$ for $1 \leq j \leq r$.  This follows from $\ee(j)
  \succeq_i \widetilde{\ee}$, implying $\ff(j) \succ_i \ff(r+1)$.

  To show that 
  \[
  \langle \overline{w_1}, \ldots, \overline{w_{r+1}} \rangle \cap
  \overline{V_1}^\perp = 0,
  \]
  we consider $\overline{w} = \sum_{j=1}^{r+1} a_j \overline{w_j} \not
  = 0$.  Put $k = \min \{ j \mid a_j \not = 0 \}$.

  \smallskip

  \noindent \textit{Case} $1$: $1 \leq k \leq r$.  Then there are
  coefficients $b_j$, for $j \not = i$, such that
  \[
  [ \overline{v_1(k)}, \overline{w} ] = [ \overline{x_{\ee(k)}} +
  \ldots, \overline{w} ] = \underbrace{a_k \cdot}_{\not = 0}
  \underbrace{[\overline{x_{\ee(k)}},\overline{y_{\ff(k)}}]}_{= z_i} +
  \sum_{j \not = i} b_j z_j \not = 0,
  \]
  because $\overline{v_1(k)}$ does not involve $\overline{x_{\ee(j)}}$
  for $j \in \{1, \ldots, r\} \smallsetminus \{k\}$ or
  $\overline{x_{\widehat{\ee}}}$ for
  $\widetilde{\ee} \succ_i \widehat{\ee}$.

  \smallskip

  \noindent \textit{Case} $2$: $k = r+1$.  In this case $\overline{w}
  = a_{r+1} \overline{y_{\ff(r+1)}}$ and
  $\overline{x_{\widetilde{\ee}}}$ occurs in $\overline{v_1(j)}$, say,
  with coefficient $c \not = 0$.  Then there are coefficients
  $b_l$, $l \not = i-1$, such that
  \[
  [ \overline{v_1(j)}, \overline{w} ] = \underbrace{c a_{r+1}
    \cdot}_{\not = 0} \underbrace{[\overline{x_{\widetilde{\ee}}},
    \overline{y_{\ff(r+1)}}]}_{= z_{i-1}} + \sum_{l \not = i-1}
  b_l z_l \not = 0. \qedhere
  \]
\end{proof}


\begin{lemma} \label{lem:autom-trival-on-centre} Let $\phi \in
  \Aut(L_K)$ with $\phi \vert_{\mathrm{Z}(L_K)} = \mathrm{id}$.  Then
  there exist $\lambda \in K^\times$ and matrices $C \in \Mat_{\lvert
    \mathbf{E} \rvert, \lvert \mathbf{F} \rvert}(K)$, $D_1 \in
  \Mat_{\lvert \mathbf{E} \rvert, n}(K)$, $D_2 \in \Mat_{\lvert
    \mathbf{F} \rvert, n}(K)$ such that, with respect to the
  basis~\eqref{equ:xyz-basis}, the automorphism $\phi$ is represented
  by the block-matrix
  \[
  \begin{pmatrix}
    \lambda \mathrm{Id}_{\lvert \mathbf{E} \rvert} & C & D_1 \\
    & \lambda^{-1} \mathrm{Id}_{\lvert \mathbf{F} \rvert} & D_2 \\
    & & \mathrm{Id}_n
  \end{pmatrix},
  \]
  where empty positions represent zeros.
\end{lemma}

\begin{proof}
  By Lemma~\ref{lem:max-ab-ideal}, there exist matrices $M_1 \in
  \Mat_{\lvert \mathbf{E} \rvert, \lvert \mathbf{E} \rvert}(K)$, $M_2
  \in \Mat_{\lvert \mathbf{F} \rvert, \lvert \mathbf{F} \rvert}(K)$,
  $C \in \Mat_{\lvert \mathbf{E} \rvert, \lvert \mathbf{F}
    \rvert}(K)$, $D_1 \in \Mat_{\lvert \mathbf{E} \rvert, n}(K)$, $D_2
  \in \Mat_{\lvert \mathbf{F} \rvert, n}(K)$ such that, with respect
  to the basis~\eqref{equ:xyz-basis}, the automorphism $\phi$ is
  represented by the block-matrix
  \[
  \begin{pmatrix}
    M_1 & C & D_1 \\
    & M_2 & D_2 \\
    &  & \mathrm{Id}_n
  \end{pmatrix}.
  \]
  We show that there exists $\lambda \in K^\times$ such that $M_1 =
  \lambda \mathrm{Id}_{\lvert \mathbf{E} \rvert}$ and $M_2 =
  \lambda^{-1} \mathrm{Id}_{\lvert \mathbf{F} \rvert}$.

  Working modulo the centre $\mathrm{Z}(L_K)$, we obtain vector spaces
  $\overline{V} = \langle \overline{x_\ee} \mid \ee \in \mathbf{E}
  \rangle$ and $\overline{W} = \langle \overline{y_\ff} \mid \ff \in
  \mathbf{F} \rangle$, equipped with bilinear forms
  \[
  B_1, \ldots, B_n \colon \overline{V} \times \overline{W} \rightarrow
  K
  \]
  such that the induced Lie bracket map satisfies
  $[\overline{v},\overline{w}] = \sum_{j=1}^n
  B_j(\overline{v},\overline{w}) z_j$
  for $\overline{v} \in \overline{V}$ and
  $\overline{w} \in \overline{W}$.

  Since $[\overline{W},\overline{W}] = 0$, we may in the following
  argument `ignore' $C$ and assume that $\phi$ maps each of
  $\overline{V}, \overline{W}$ to itself.  Since $\phi$ restricts to
  the identity on $\mathrm{Z}(L_K)$, we deduce that $\phi$ preserves
  the bilinear forms $B_1, \ldots, B_n$.  Below we show that
  $\phi \vert_{\overline{V}} = \lambda \mathrm{id}$ for some
  $\lambda \in K^\times$.  This implies
  $\phi \vert_{\overline{W}} = \lambda^{-1} \mathrm{id}$, since the
  intersection of the (right) radicals
  \[
  \overline{R_i} = \Rad_{\overline{W}}(B_i) = \langle y_\ff \mid \ff =
  (f_1,\ldots,f_n) \in \mathbf{F} \text{ with } f_i = 0 \rangle
  \]
  for $i \in \{1,\ldots,n\}$ is trivial.

  For $m = 1$ we have $\lvert \mathbf{E} \rvert = 1$ and there is
  nothing further to show.  Now suppose that $m \ge 2$, and fix
  $i \in \{1,\ldots,n\}$. The radical $\overline{R_i}$ is a
  $\phi$-invariant subspace of~$\overline{W}$.  Consequently, also
  \begin{align*}
    \overline{U_i} = \bigcap\nolimits_{j=1}^n \overline{R_i}^{\perp,j}
    & = \{ \overline{v} \in \overline{V} \mid \forall j \in
    \{1,\ldots,n\} :
    B_j(\overline{v},\overline{R_i}) = 0 \} \\
    & = \langle \overline{x_\ee} \mid \ee = (e_1,\ldots,e_n) \in
    \mathbf{E} \text{ with } e_i \geq 1 \rangle
  \end{align*}
  is $\phi$-invariant.

  Put $\overline{V_1} = \overline{U_i}$ and $\overline{W_1} =
  \overline{W}/\overline{R_i}$, equipped with the induced bilinear
  forms
  \[
  \widetilde{B_1}, \ldots, \widetilde{B_n} \colon \overline{V_1} \times
  \overline{W_1} \rightarrow K
  \]
  and an induced automorphism $\widetilde{\phi}$.  By induction on
  $m$, we conclude that there exists $\lambda_i \in K^\times$ such that
  \[
  \phi \vert_{\overline{U_i}} = \widetilde{\phi}
  \vert_{\overline{V_1}} = \lambda_i \mathrm{id}
  \]
  and, though we will not use this, $\widetilde{\phi}
  \vert_{\overline{W_1}} = \lambda_i^{\, -1} \mathrm{id}$, i.e.,
  $\widetilde{\phi} \vert_{\overline{W}} \equiv \lambda_i^{\, -1}
  \mathrm{id}$ modulo $\overline{R_i}$.

  Repeating this argument for different $i \in \{1,\ldots,n\}$ and
  comparing on the intersection of the $\overline{U_i}$, we deduce
  that $\phi \vert_{\overline{V}} = \lambda \mathrm{id}$ for a common
  $\lambda \in K^\times$. 
\end{proof}

\begin{lemma} \label{lem:autom-unipotent} Let $A = \langle y_\ff \mid
  \ff \in \mathbf{F} \rangle + \mathrm{Z}(L_K) \trianglelefteq L_K$, as in
  Lemma~\textup{\ref{lem:max-ab-ideal}}.  Suppose that $\phi \in \Aut(L_K)$
  induces the identity map on $L_K/A$, $A/\mathrm{Z}(L_K)$ and
  $\mathrm{Z}(L_K)$.  Then there exist matrices $C \in \Mat_{\lvert
    \mathbf{E} \rvert, \lvert \mathbf{F} \rvert}(K)$, $D_1 \in
  \Mat_{\lvert \mathbf{E} \rvert, n}(K)$, $D_2 \in \Mat_{\lvert
    \mathbf{F} \rvert, n}(K)$ such that, with respect to the
  basis~\eqref{equ:xyz-basis}, the automorphism $\phi$ is represented
  by the block-matrix
  \begin{equation} \label{equ:block-mat}
    \begin{pmatrix}
      \mathrm{Id}_{\lvert \mathbf{E} \rvert} & C & D_1 \\
      & \mathrm{Id}_{\lvert \mathbf{F} \rvert} & D_2 \\
      & & \mathrm{Id}_n
    \end{pmatrix},
  \end{equation}
  where empty positions represent zeros.  The rows of the matrix $C$
  are naturally indexed by $\ee \in \mathbf{E}$; its columns are
  naturally indexed by $\ff \in \mathbf{F}$.  Writing $C =
  (c_{\ee,\ff})_{(\ee,\ff) \in \mathbf{E} \times \mathbf{F}}$, there
  are parameters $b_\gg$ indexed by the elements of
  \[
  \mathbf{G} = \{ \gg \mid \gg = (g_1, \ldots, g_n) \in \N_0^n
  \text{ with $g_1 + \ldots + g_n = 2m-1$} \}
  \]
  such that $c_{\ee,\ff} = b_{\ee+\ff}$ for all $\ee \in \mathbf{E}$,
  $\ff \in \mathbf{F}$.

  Conversely, given matrices $C = (c_{\ee,\ff})_{(\ee,\ff) \in
    \mathbf{E} \times \mathbf{F}} \in \Mat_{\lvert \mathbf{E} \rvert,
    \lvert \mathbf{F} \rvert}(K)$, $D_1 \in \Mat_{\lvert \mathbf{E}
    \rvert, n}(K)$, $D_2 \in \Mat_{\lvert \mathbf{F} \rvert, n}(K)$
  such that $c_{\ee,\ff} = c_{\ee',\ff'}$ for $\ee,\ee' \in
  \mathbf{E}$ and $\ff,\ff' \in \mathbf{F}$ with $\ee + \ff = \ee' +
  \ff'$ there is an automorphism of $L_K$ that is represented by the
  block matrix~\eqref{equ:block-mat}.
\end{lemma}

\begin{proof}
  We indicate how to prove the first part; the second part is a
  routine verification.  The assumptions on $\phi$ show directly that
  $\phi$ is represented by a block matrix of the
  form~\eqref{equ:block-mat}.  In particular, this means that, for
  each $\ee \in \mathbf{E}$,
  \[
  x_\ee \phi \equiv x_\ee + \sum\nolimits_{\ff \in \mathbf{F}} c_{\ee,\ff}
  y_\ff  \pmod{Z(L_K)}.
  \]
  Since $\phi$ is an automorphism, we conclude that, for $\ee,\ee' \in
  \mathbf{E}$,
  \[
  \big[x_\ee + \sum\nolimits_{\ff \in \mathbf{F}} c_{\ee,\ff} \,
  y_\ff, x_{\ee'} + \sum\nolimits_{\ff' \in \mathbf{F}} c_{\ee',\ff'}
  \, y_{\ff'} \big] = [x_\ee\phi,x_{\ee'}\phi] = [x_\ee,x_{\ee'}]\phi
  = 0,
  \]
  and consequently $c_{\ee,\ff} = c_{\ee',\ff'}$, for $\ff,\ff' \in
  \mathbf{F}$, whenever
  \begin{equation} \label{equ:1-step} \ff - \ee' = \ff' - \ee =
    (\underbrace{0,\ldots,0}_{i-1 \text{
        entries}},1,\underbrace{0,\ldots,0}_{n-i \text{ entries}})
    \quad \text{for some $i \in \{1,\ldots,n\}$}.
  \end{equation}
  We claim that this implies -- and is thus equivalent to -- the
  condition
  \[
  c_{\ee,\ff} = c_{\ee',\ff'} \quad \text{whenever $\ee+\ff =
    \ee'+\ff'$.}
  \]
  Indeed, given any $\ee = (e_1,\ldots,e_n), \ee'=(e_1',\ldots,e_n')
  \in \mathbf{E}$ and $\ff,\ff' \in \mathbf{F}$ with $\ee+\ff =
  \ee'+\ff'$ one can find a chain
  \[
  (\ee,\ff) = (\ee_1,\ff_1), (\ee_2,\ff_2), \ldots,
  (\ee_\delta,\ff_\delta), (\ee_{\delta+1},\ff_{\delta+1}) =
  (\ee',\ff')
  \]
  of length $\delta = \sum_{i=1}^n \lvert e_i - e_i' \rvert$ such that
  the condition~\eqref{equ:1-step} is satisfied for any two
  consecutive terms $(\ee_t,\ff_t), (\ee_{t+1},\ff_{t+1})$ of the
  chain so that
  \[
  c_{\ee,\ff} = c_{\ee_1,\ff_1} = c_{\ee_2,\ff_2} = \ldots =
  c_{\ee_{\delta},\ff_{\delta}} =
  c_{\ee_{\delta+1},\ff_{\delta+1}} = c_{\ee',\ff'}. 
  \]   
  To produce such a chain, we observe that $\delta \equiv_2 0$ so that
  we can move inductively in steps of two.  Suppose we have reached
  $(\ee_t,\ff_t)$, not yet equal to $(\ee',\ff')$.  Writing
  $\ee_t = (e_{t,1},\ldots,e_{t,n})$, we locate
  $j,k \in \{1, \ldots, n\}$ such that $e_{t,j} < e_j'$ and
  $e_{t,k} > e_k'$. Using \eqref{equ:1-step} at $j$ and $k$ in place
  of~$i$, we reach via some intermediate $(\ee_{t+1},\ff_{t+1})$ the
  next term $(\ee_{t+2},\ff_{t+2})$, where
  $\ee_{t+2} = (e_{t+2,1},\ldots,e_{t+2,n})$ satisfies:
  $e_{t+2,j} =e_{t,j}+1$, $e_{t+2,k}=e_{t,k}-1$ and
  $e_{t+2,i} = e_{t,i}$ for all remaining indices~$i$.  As
  $\ee_t + \ff_t$ remains constant throughout, at the end of the final
  step $\ee_{\delta+1} = \ee'$ automatically implies
  $\ff_{\delta+1} = \ff'$.
\end{proof}

\begin{lemma} \label{lem:reductive-part} Suppose that the
  characteristic of $K$ is $0$.  Then every
  $\psi \in \GL(\mathrm{Z}(L_K))$ extends to an automorphism
  $\phi \in \Aut(L_K)$ so that $\phi \vert_{\mathrm{Z}(L_K)} = \psi$.
  Moreover, $\phi$ can be chosen so that the spaces
  $\langle x_\ee \mid \ee \in \mathbf{E} \rangle$ and
  $\langle y_\ff \mid \ff \in \mathbf{F} \rangle$ are $\phi$-invariant
  and the action of $\phi$ on them corresponds to induced actions of
  $\psi$ on symmetric powers of the dual space of $\mathrm{Z}(L_K)$
  and on symmetric powers of $\mathrm{Z}(L_K)$, respectively.
\end{lemma}

\begin{proof}
  Regard $\mathrm{Z}(L_K) = \langle z_1, \ldots, z_n \rangle$ as the dual space
  $V^\vee$ of an $n$-dimensional $K$-vector space $V = \langle \xi_1,
  \ldots, \xi_n \rangle$ such that, for $i,j \in \{1,\ldots,n\}$,
  \[
  (\xi_i)z_j = 
  \begin{cases}
    1 & \text{if $i=j$,} \\
    0 & \text{otherwise.}
  \end{cases}
  \]
  The $(m-1)$th symmetric power $S^{m-1}V$ of $V$ can be constructed
  as $V^{\otimes \, m-1}/\mathrm{Sym}(m-1)$, where
  $(\nu_1 \otimes \ldots \otimes \nu_{m-1})^\sigma = \nu_{1
    \sigma^{-1}} \otimes \ldots \otimes \nu_{(m-1)\sigma^{-1}}$
  for $\sigma \in \mathrm{Sym}(m-1)$.  Writing
  $\nu_1 \cdots \nu_{m-1}$ for the element of $S^{m-1}V$
  represented by $\nu_1 \otimes \ldots \otimes \nu_{m-1}$, one
  easily sees that a basis of $S^{m-1}V$ is given by the vectors
  $\xi_1^{\, e_1} \cdots \xi_n^{\, e_n}$,
  $\ee = (e_1,\ldots,e_n) \in \mathbf{E}$.  In a similar way, we form
  the $m$th symmetric power $S^m(V^\vee)$ with basis
  $z_1^{\, f_1} \cdots z_n^{\, f_n}$, $\ff \in \mathbf{F}$.

  One can interpret $S^m(V^\vee)$ as $(S^m V)^\vee$, and this suggests
  a change of basis.  Indeed, the natural map
  \[
  V^{\otimes \, m} \times (V^\vee)^{\otimes \, m} \rightarrow K, \quad
  \text{induced by $(\nu_1 \otimes \ldots \otimes \nu_m)(a_1 \otimes
    \ldots \otimes a_m) = \prod_{i=1}^m (\nu_i) a_i$}
  \]
   gives rise to an
  isomorphism $(V^\vee)^{\otimes \, m} \cong (V^{\otimes \,
    m})^\vee$.  Now $\xi_1^{\, f_1'} \cdots \xi_n^{\, f_n'} \in S^mV$ is
  represented, for instance, by
  \[
  \underbrace{\xi_1 \otimes \ldots \otimes \xi_1}_{f_1' \text{ factors}} \otimes
  \ldots \otimes \underbrace{\xi_n \otimes \ldots \otimes
    \xi_n}_{f_n' \text{ factors}} \in V^{\otimes \, m}.
  \] 
  Moreover, as $K$ has characteristic $0$, there is an embedding
  \[
  S^m(V^\vee) \hookrightarrow (V^\vee)^{\otimes \, m}, \quad z_1^{\,
    f_1} \cdots z_n^{\, f_n} \mapsto \sum_{\sigma \in \mathrm{Sym}(m)}
  (\underbrace{z_1 \otimes \ldots \otimes z_1}_{f_1 \text{ factors}}
  \otimes \ldots \otimes \underbrace{z_n \otimes \ldots \otimes
    z_n}_{f_n \text{ factors}})^\sigma;
  \] 
  the image of this embedding consists of the elements of
  $(V^\vee)^{\otimes \, m}$ that are fixed under the action of
  $\mathrm{Sym}(m)$.  Hence we can evaluate (the image of) $z_1^{\,
    f_1} \cdots z_n^{\, f_n} \in S^m(V^\vee)$ (under this embedding)
  at $\xi_1^{\, f_1'} \cdots \xi_n^{\, f_n'} \in S^m V$ to obtain
  \[
  (\xi_1^{\, f_1'} \cdots \xi_n^{\, f_n'}) \, (z_1^{\, f_1} \cdots z_n^{\,
    f_n}) =
  \begin{cases}
    \prod\nolimits_{i=1}^n (f_i !) & \text{if $(f_1',\ldots,f_n') =
      (f_1,\ldots,f_n)$,} \\
    0 & \text{otherwise}
  \end{cases}
  \]
  Therefore the basis $\xi_1^{\, f_1'} \cdots \xi_n^{\, f_n'}$,
  indexed by $\mathbf{f}' = (f_1',\ldots,f_n') \in \mathbf{F}$, of
  $S^mV$ gives rise to the dual basis $(\prod_{i=1}^n (f_i !))^{-1} \,
  z_1^{\, f_1} \cdots z_n^{\, f_n}$, indexed by $\mathbf{f} =
  (f_1,\ldots,f_n) \in \mathbf{F}$, of $S^m(V^\vee)$, and this is the
  basis that we prefer to work with in the present context.

  We define a multilinear map
  \[
  \Phi \colon S^{m-1}V \times \underbrace{V}_{=S^1V} \times
  S^m(V^\vee) \rightarrow K
  \]
  by setting
  \begin{align*}
    \Phi \left( \xi_1^{\, e_1} \cdots \xi_n^{\, e_n}, \xi_k, \big(
      \prod\nolimits_{i=1}^n (f_i !) \big)^{-1} \, z_1^{\, f_1} \cdots
      z_n^{\, f_n} \right) & = \big( \xi_1^{\, \widetilde{e_1}} \cdots
    \xi_n^{\, \widetilde{e_n}} \big) \left(
      \big(\prod\nolimits_{i=1}^n
      (f_i !) \big)^{-1} \, z_1^{\, f_1} \cdots z_n^{\, f_n} \right) \\
    & =
    \begin{cases}
      1 & \text{if $(\widetilde{e_1},\ldots,\widetilde{e_n}) =
        (f_1,\ldots,f_n)$,} \\
      0 & \text{otherwise},
    \end{cases}
  \end{align*}
  where $\widetilde{e_k} = e_k+1$ and $\widetilde{e_j} = e_j$ for $j
  \not = k$, and extending linearly in each of the three arguments.
  This map $\Phi$ connects to the original Lie algebra $L_K$ as follows.
  For $\ee = (e_1,\ldots,e_n) \in \mathbf{E}$ and $\ff =
  (f_1,\ldots,f_n) \in \mathbf{F}$,
  \begin{multline*}
    [x_\ee,y_\ff] = \sum\nolimits_{k=1}^n \Phi \left( \xi_1^{\, e_1}
      \cdots \xi_n^{\, e_n}, \xi_k, \big( \prod\nolimits_{i=1}^n (f_i
      !) \big)^{-1} \, z_1^{\, f_1} \cdots z_n^{\, f_n} \right) \cdot
    z_k \\
    = \Phi \left( \xi_1^{\, e_1} \cdots \xi_n^{\, e_n},
      \boldsymbol{\cdot} , \big( \prod\nolimits_{i=1}^n (f_i !)
      \big)^{-1} \, z_1^{\, f_1} \cdots z_n^{\, f_n} \right) \in
    V^\vee = \mathrm{Z}(L_K).
  \end{multline*}

  The fact that $\Phi$ is multilinear means that we can monitor
  changes to a set of structure constants for $L_K$ due to changes to
  any of the three parts of the basis
  \[
  x_\ee, \, \ee \in \mathbf{E}, \quad y_\ff, \, \ff \in \mathbf{F},
  \quad z_j, \, j \in \{1, \ldots, n\}.
  \]
  In particular, given any $\psi \in \GL(V^\vee)$, there are the
  natural induced actions of $\psi$ on
  \begin{itemize}
  \item[$\circ$] $V^\vee$ via the natural representation,
  \item[$\circ$] $V$ via the contragredient representation
    (given in matrices by inverse transpose) which we denote by
    $\psi^*$,
  \item[$\circ$] $S^m(V^\vee)$, sending $(\prod_{j=1}^n (f_j!))^{-1}
    z_1^{\, f_1} \cdots z_n^{\, f_n}$ to $(\prod_{j=1}^n (f_j!))^{-1}
    (z_1 \psi)^{\, f_1} \cdots (z_n \psi)^{\, f_n}$,
  \item[$\circ$] $S^{m-1} V$, sending $\xi_1^{\, e_1} \cdots \xi_n^{\,
      e_n}$ to $(\xi_1 \psi^*)^{\, e_1} \cdots (\xi_n \psi^*)^{\,
      e_n}$.
  \end{itemize}
  In this set-up we clearly obtain
  \begin{multline*}
    \Phi \left( \xi_1^{\, e_1} \cdots \xi_n^{\, e_n}, \xi_k, \big(
      \prod\nolimits_{i=1}^n (f_i !) \big)^{-1} \, z_1^{\, f_1} \cdots
      z_n^{\, f_n} \right) \\
    = \Phi \left( (\xi_1 \psi^*)^{\, e_1} \cdots (\xi_n \psi^*)^{\, e_n},
      (\xi_k \psi^*), \big( \prod\nolimits_{i=1}^n (f_i !) \big)^{-1} \,
      (z_1 \psi)^{\, f_1} \cdots (z_n \psi)^{\, f_n} \right),
  \end{multline*}
  because the bases $\xi_1 \psi^*, \ldots, \xi_n \psi^*$ of $V$ and
  $z_1 \psi, \ldots, z_n \psi$ of $V^\vee$ are dual to one another by
  construction.  Thus the given $\psi \in \GL(V^\vee) = \GL(\mathrm{Z}(L_K))$
  extends to an automorphism $\phi \in \Aut(L_K)$ so that $\phi
  \vert_{\mathrm{Z}(L_K)} = \psi$; moreover, $\phi$ acts on the spaces $\langle
  x_\ee \mid \ee \in \mathbf{E} \rangle$ and $\langle y_\ff \mid \ff
  \in \mathbf{F} \rangle$ as $\psi$ does on the corresponding
  symmetric powers.
\end{proof}

From Lemmata~\ref{lem:autom-trival-on-centre},
\ref{lem:autom-unipotent} and \ref{lem:reductive-part}, including the
proof of Lemma~\ref{lem:reductive-part}, we deduce the following
result.

\begin{theorem} \label{thm:alg-autom-grp} Let $m,n \in \N$ with $n
  \geq 2$, and let $L = L_{m,n}$ be the class-$2$ nilpotent $\Z$-Lie
  lattice introduced in Definition~\textup{\ref{def-Lmn}}, so that
  $\mathrm{rank}_\Z L = d = r_1 + r_2 + r_3$, where
  \[
  r_1 = \lvert \mathbf{E} \rvert = \binom{m+n-2}{n-1}, \qquad r_2 =
  \lvert \mathbf{F} \rvert = \binom{m+n-1}{n-1}, \qquad r_3 = n.
  \]
  Then the algebraic automorphism group $\mathsf{Aut}(L)$, regarded as
  an affine $\Z$-group scheme $\aG \subseteq \GL_d$ with respect to
  the basis~\eqref{equ:xyz-basis}, has the following structure: $\aG$
  decomposes over $\Z$ as a semidirect product $\aG = \aN \rtimes \aH$
  of its unipotent radical $\aN$ and a reductive complement~$\aH$,
  both defined over $\Z$, such that

  \smallskip

  $\circ \phantom{x}$ $\aN$ consists of elements of the form
  \[
    \begin{pmatrix}
      \mathrm{Id}_{\lvert \mathbf{E} \rvert} & C & D_1 \\
      & \mathrm{Id}_{\lvert \mathbf{F} \rvert} & D_2 \\
      & & \mathrm{Id}_n
    \end{pmatrix},
  \]
  where $C \in \Mat_{\lvert \mathbf{E} \rvert, \lvert \mathbf{F}
    \rvert}$, $D_1 \in \Mat_{\lvert \mathbf{E} \rvert, n}$, $D_2 \in
  \Mat_{\lvert \mathbf{F} \rvert, n}$ subject to the following
  conditions.  The rows of $C$ are naturally indexed by $\ee \in
  \mathbf{E}$, its columns are naturally indexed by $\ff \in
  \mathbf{F}$.  Writing $C = (c_{\ee,\ff})_{(\ee,\ff) \in \mathbf{E}
    \times \mathbf{F}}$, there are parameters $b_\gg$ indexed by the
  elements of
  \[
  \mathbf{G} = \{ \gg \mid \gg = (g_1, \ldots, g_n) \in \N_0^n \text{
    with $g_1 + \ldots + g_n = 2m-1$} \}
  \]
  such that $c_{\ee,\ff} = b_{\ee+\ff}$ for all $(\ee,\ff) \in
  \mathbf{E} \times \mathbf{F}$.

  \smallskip

  $\circ \phantom{x}$ $\aH \cong_\Z \GL_n \times \GL_1$ consists of
  elements of the form
  \[
  \begin{pmatrix} \lambda\rho_1(A) &  & \\
    & \lambda^{-1} \rho_2(A) & \\
    & & A
  \end{pmatrix},
  \]
  where $A \in \GL_n$, $\lambda \in \aGm$ and $\rho_1, \rho_2$ are
  suitable symmetric power representations of $\GL_n$ in dimensions
  $r_1$ and $r_2$.
\end{theorem}


\section{Evaluating $p$-adic integrals via the Bruhat decomposition}
\label{sec:evaluate-p-adic-integrals}

In~\cite[Sections~2]{dSLu96}, du Sautoy and Lubotzky provide a general
treatment of $\fp$-adic integrals of the
form~\eqref{equ:zeta-alg-group}, subject to several simplifying
assumptions.  We now give an outline of the procedure and indicate how
it applies to the specific affine group scheme supplied by
Theorem~\ref{thm:alg-autom-grp}.  It will turn out that for our
application all necessary assumptions are automatically satisfied.

Let $\aG \subseteq \GL_d$ be an affine group scheme over the ring of
integers $\fo$ of a number field~$\fk$.  Decompose the connected
component of the identity as a semidirect product $\aG^\circ = \aN
\rtimes \aH$ of the unipotent radical $\aN$ and a reductive
complement~$\aH$.  For a finite prime $\fp$ let $\fk_{\fp}$ denote the
completion at $\fp$, and let $\fo_\fp$ denote the valuation ring
of~$\fk_\fp$.  Fix a uniformising parameter $\pi$ for $\fo_\fp$, and
let $q$ denote the size of the residue field $\fo_\fp / \pi \fo_\fp$.
Setting $G = \aG(\fk_\fp)$ and $G^+ = G \cap \Mat_d(\fo_\fp)$, we are
interested in the zeta function
\[
\mathcal{Z}_{\aG,\fp}(s) = \int_{G^+} \lvert \det(g) \rvert_\fp^s \,
d\mu_G (g).
\]

The first simplifying assumption is that $G = \aG(\fo_\fp) \,
\aG^\circ(\fk_\fp)$, which allows one to work with the connected group
$\aG^\circ$ rather than~$\aG$.  In our application, the group $\aG$ is
already connected and we continue the discussion under the
assumption~$\aG = \aG^\circ$.  We write
\[
N = \aN(\fk_\fp) \quad \text{and} \quad H = \aH(\fk_\fp).
\]
Assume further that $G \subseteq \GL_d(\fk_\fp)$ is in block form,
where $H$ is block diagonal and $N$ is block upper unitriangular in
the following sense.  There is a partition $d = r_1 + \ldots + r_c$
such that, setting $s_i = r_1 + \ldots + r_{i-1}$ for $i \in
\{1,\ldots,c+1\}$,
\begin{enumerate}
\item[$\circ$] the vector space $V = \fk_\fp^d$ on which $G$ acts from
  the right decomposes into a direct sum of $H$-stable subspaces $U_i
  = \{ (0,\ldots,0) \} \times \fk_\fp^{r_i} \times \{ (0,\ldots,0)
  \}$, where the vectors $(0,\ldots,0)$ have $s_i$, respectively
  $d-s_{i+1}$, entries;
\item[$\circ$] setting $V_i = U_i \oplus \ldots \oplus U_c$, each
  $V_i$ is $N$-stable and $N$ acts trivially on $V_i/V_{i+1}$.
\end{enumerate}
In our application, Theorem~~\ref{thm:alg-autom-grp} provides such a
block decomposition with $c=3$.  For each $i \in \{2,\ldots,c+1\}$ let
$N_{i-1} = N \cap \ker(\psi'_i)$, where $\psi'_i \colon G \rightarrow
\Aut(V/V_i)$ denotes the natural action.  Let $\psi_i \colon G/N_{i-1}
\rightarrow \Aut(V/V_i)$ denote the induced map, and define
\[
(G/N_{i-1})^+ = \psi_i^{-1} \big(\psi_i(G/N_{i-1}) \cap
\Mat_{s_i}(\fo_\fp) \big).
\]
Setting $H^+ = H \cap \Mat_d(\fo_\fp)$, we can now state (a slightly
simplified version of) the `lifting condition', which cannot in
general be satisfied by moving to an equivalent representation; see
\cite[p.~6]{Be11}.

\begin{condition}[cf.\
  {\cite[Assumption~2.3]{dSLu96}}]\label{lifting-condition}
  For each $i \in \{3,\ldots,c \}$ and every $g_0 N_{i-1} \in
  (NH^+/N_{i-1})^+$ there exists $g \in G^+$ such that $g_0 N_{i-1} = g
  N_{i-1}$.
\end{condition}

In our application, $c = 3$ is so small that the `lifting condition'
holds for trivial reasons: according to the description given in
Theorem~\ref{thm:alg-autom-grp}, the relevant blocks $D_1$ and $D_2$
of any element of $N$ can just be replaced by zero blocks to achieve a
lifting.

For $i \in \{2,\ldots,c\}$ there is a natural embedding of
$N_{i-1}/N_i \hookrightarrow (V_i/V_{i+1})^{s_i}$ via the action of
$N_{i-1}/N_i$ on $V/V_{i+1}$, recorded on the natural basis.  The
action of $H$ on $V_i/V_{i+1}$ induces, for each $h \in H$, a map
\[
\tau_i(h) \colon N_{i-1}/N_i \hookrightarrow (V_i/V_{i+1})^{s_i}.
\]
Define $\theta_{i-1} \colon H \rightarrow \mathbb{R}_{\ge 0}$ by setting
\[
\theta_{i-1}(h) = \mu_{N_{i-1}/N_i} \big( \{n N_i \in N_{i-1}/N_i \mid (n
N_i) \tau_i(h) \in \Mat_{s_i,r_i}(\fo_\fp)\} \big),
\]  
where $\mu_{N_{i-1}/N_i}$ denotes the right Haar measure on
$N_{i-1}/N_i$, normalised such that the set $\psi_{i+1}^{-1}
(\psi_{i+1}(N_{i-1}/N_i) \cap \Mat_{s_{i+1}}(\fo_\fp))$ has
measure~$1$.  In our application, $c=3$ so that there are two
functions, $\theta_1$ and $\theta_2$, that should be computed.

Write $\mu_G$, respectively $\mu_N$ and $\mu_H$, for the right Haar
measure on~$G$, respectively $N$ and $H$, normalised such that
$\mu_G(\aG(\fo_\fp)) = 1$, respectively $\mu_N( \aN(\fo_\fp)) = 1$ and
$\mu_H( \aH(\fo_\fp)) = 1$.  From $G = N \rtimes H$ one deduces that
$\mu_G = \prod_{i=2}^c \mu_{N_{i-1}/N_i} \cdot \mu_H$.  Writing
$\theta(h) = \mu_N (\{ n \in N \mid nh \in G^+ \})$ for $h \in H$, we
can now state the following result of du Sautoy and Lubotzky.

\begin{theorem}[{\cite[Theorem~2.2]{dSLu96}}] \label{thm:dS-Lu-weak-version}
  In the set-up described above, the various assumptions, including
  the `lifting condition', guarantee that
  $\theta(h) = \prod_{i=1}^{c-1} \theta_i(h)$ for $h \in H$, and
  consequently,
  \[
  \mathcal{Z}_{\aG,\fp}(s) = \mathcal{Z}_{\aH,\fp,\theta}(s), \qquad
  \text{where} \quad \mathcal{Z}_{\aH,\fp,\theta}(s) = \int_{H^+}
  \lvert \det(h) \rvert_\fp^s \, \prod_{i=1}^{c-1} \theta_i(h) \,
  d\mu_H(h).
  \]
\end{theorem}

Next we recall the machinery developed by Igusa~\cite{Ig89} as well as
du Sautoy and Lubotzky~\cite{dSLu96} for utilising a $\fp$-adic Bruhat
decomposition in order to compute integrals over reductive $\fp$-adic
groups; a useful reference for practical purposes is~\cite{Be11}.

We assume that the reductive $\fp$-adic group $H$ is split over
$\fk_\fp$ and fix a maximal $\fk_p$-split torus $T = \aT(\fk_\fp)$
in~$H$.  In our application, $\aH \cong_\Z \GL_n \times \GL_1$ is
clearly split and we take the subgroup of diagonal matrices for~$\aT$.
Elements of $\aT$ act by conjugation on minimal closed unipotent
subgroups of $\aH$ which correspond to elements of the root system
$\Phi$ associated to~$\aH$.  Under this action, the (finite) Weyl
group corresponds to $W = \mathrm{N}_\aH(\aT)/\aT$; in our
application, $W \cong \mathrm{Sym}(n)$.  We suppress here
some necessary requirements of `good reduction modulo $\fp$' in
identifying $\aT$ with $\aGm^{\, \dim \aT}$ and identifying each root
subgroup with the additive group~$\aGa$; these technical requirements
hold trivially in our application.  Roots of $\aH$ relative to $\aT$
are identified with elements of $\Hom(\aT, \aGm)$ via the conjugation
action of $\mathrm{N}_\aH(\aT)/\aT$ on root subgroups.  We choose a
set of simple roots $\alpha_1,\ldots, \alpha_l$ which define the
positive roots~$\Phi^+$; in our application, $l = n-1$.  We denote
by $\Xi = \Hom(\aGm, \aT)$ the set of cocharacters of~$\aT$, and we
recall the natural pairing between characters and cocharacters,
\[
\Hom(\aT, \aGm) \times \Hom(\aGm, \aT) \rightarrow \Z, \quad (\beta,
\xi) \mapsto \langle \beta, \xi \rangle,
\]
defined by $\beta(\xi(\tau)) = \tau^{\langle \beta, \xi \rangle}$ for
all $\tau\in \aGm$.  

We refer to~\cite[p.~75]{dSLu96} for a description of Iwahori
subgroup~$\mcalB$ of the reductive $\fp$-adic group~$H$.  Recalling
that $\pi$ denotes a uniformising parameter for $\fo_\fp$, we state
the $\fp$-adic Bruhat decompositions of Iwahori and
Matsumoto~\cite{IwMa65}:
\[
H = \coprod_{\substack{w\in W \\ \xi \in \Xi}} \mcalB \, w \xi(\pi) \,
\mcalB \qquad \text{and} \qquad \aH(\fo_\fp) = \coprod_{w\in W} \mcalB
w \mcalB
\] 
where elements of $W$ are suitably interpreted as representatives
in~$\mathrm{N}_H(T)$.  One defines $\Xi^+ = \{\xi \in \Xi \mid
\xi(\pi) \in \aH(\fo_\fp) \}$ and
\begin{multline*}
  w \Xi_w^+ = \big\{ \xi \in \Xi^+ \mid \alpha_i(\xi(\pi)) \in \fo_\fp
  \text{ for all } i \in \{1, \ldots, l \}, \\ \text{ and }
  \alpha_i(\xi(\pi)) \in \pi \fo_\fp \text{ whenever } \alpha_i\in
  w(\Phi^-) \big\},
\end{multline*}
where $\Phi^-$ denotes the set of negative roots.  As in \cite[Section
5.2]{Be11}, it will turn out to be convenient to judiciously choose a
basis for $\Hom(\aT, \aGm)$ consisting of simple roots
$\alpha_1,\ldots,\alpha_l$ and inverses $\omega_i^{\, -1}$ of the
dominant weights for the contragredient representations of $H$ acting
on~$V_i/V_{i+1}$.  Using the dual basis for $\Xi = \Hom(\aGm,\aT)$, we
thereby obtain an explicit description of the sets~$w\Xi_w^+$.

Utilising symmetries in the affine Weyl group and the fact that the
functions $\lvert \det \rvert_\fp$, $\theta$ are constant on double
cosets of the Iwahori subgroup $\mcalB \leq \aH(\fo_\fp)$, the first
author proved the following proposition, generalising results of du Sautoy
and Lubotzky~\cite{dSLu96} and Igusa~\cite{Ig89}.

\begin{proposition}[{\cite[Proposition~4.2]{Be11}}] 
  \label{proposition:formula-integral}
  Suppose that $\aH$ has a $\fk$-split torus~$\aT$.  Under good
  reduction assumptions, which are satisfied for almost all primes
  $\fp$,
  \[
  \mathcal{Z}_{\aH,\fp,\theta}(s) = \sum_{w \in W} q^{-\ell(w)} \sum_{\xi \in
    w \Xi_w^+} q^{\langle \prod_{\beta \in \Phi^+} \beta, \xi \rangle}
  \, \lvert \det(\xi(\pi)) \rvert_\fp^s \, \theta(\xi(\pi))
  \]
  where $\ell$ denotes the standard length function on the Weyl
  group~$W$.
\end{proposition}

In our application, we have, in fact, good reduction for all primes.

\begin{remark} \label{rem:funct-equa-half-way} Although we have not
  yet calculated the pro-isomorphic zeta functions of
  $\fo_\fp \otimes_\Z L_{m,n}$, we are already able to deduce that the
  former will satisfy functional equations. We do so by invoking
  \cite[Theorem~1.1]{Be11}. Indeed, one can define a faithful
  representation $\rho \colon \GL_n\times \GL_1\to \aH$ given by the
  description of the algebraic automorphism groups in
  Theorem~\ref{thm:alg-autom-grp}.  This representation preserves
  integrality and preserves the Haar measure.  The conditions for
  Theorem~1.1 in \cite{Be11} are that the group $\aH$ should split
  over $\fk$, that the `lifting condition' should hold and that the
  number $r$ of irreducible components of the representation should
  equal the dimension $d$ of a maximal central torus of
  $\GL_n\times \GL_1$.  The first two conditions hold.  The third
  condition does not hold as we have $r=3$, $d=2$.  However, since the
  third block corresponds to the action on the centre and every
  element of the centre is obtained as a commutator of suitable
  generators, it is possible to express the dominant weight of the
  contragredient representation for the last block $A$ as a positive
  linear combination of weights occurring in the first two blocks
  $\lambda \rho_1(A), \lambda^{-1} \rho_2(A)$ (and hence as a positive
  linear combination of the dominant weights for the first two blocks,
  together with simple roots in the root system associated to $\aH$
  relative to a choice of maximal torus). In this way, all the
  requirements used to prove a functional equation in Section 5.2
  of~\cite{Be11} are fulfilled, and the argument there implies the
  existence of functional equations in our situation.
\end{remark}


\section{Computation of the local zeta
  functions} \label{sec:explicit-computation}

In order to establish Theorem~\ref{theorem:main} it remains to carry
out the detailed computations, following the framework described in
Section~\ref{sec:evaluate-p-adic-integrals}.  Let $m,n \in \N$ with
$n \geq 2$, and put $d = r_1 + r_2 + r_3$, where
\[
r_1 = \lvert \mathbf{E} \rvert = \binom{m+n-2}{n-1}, \qquad r_2 =
\lvert \mathbf{F} \rvert = \binom{m+n-1}{n-1}, \qquad r_3 = n.
\]
We consider the affine group scheme $\aG = \aN \rtimes \aH \subseteq
\GL_d$ described in Theorem~\ref{thm:alg-autom-grp}.  Putting $\aH_0 =
\GL_n \times \GL_1$, the group $\aH$ is the image of the faithful
representation
\begin{equation} \label{equ:elements-H} \rho \colon \aH_0
  \hookrightarrow \GL_d, \qquad (A,\lambda) \mapsto
  \begin{pmatrix} \lambda\rho_1(A) &  & \\
    & \lambda^{-1} \rho_2(A) & \\
    & & A
  \end{pmatrix},
\end{equation}
where $\rho_1, \rho_2$ are the symmetric power representations of
$\GL_n$ described in the proof of Lemma~\ref{lem:reductive-part}.  We
denote by $\aT_0$ the maximal torus of $\aH_0$ consisting of diagonal
matrices.  Via this choice, the root system $\Phi$ of $\aH_0$ has
simple roots $\alpha_1, \ldots, \alpha_{n-1}$ given by
\[
\alpha_i(t) = a_i/a_{i+1} \qquad \text{for $1 \leq i \leq n-1$ and $t
  = (\diag(a_1,\ldots,a_n),\lambda) \in \aT_0$.}
\] 
The corresponding Weyl group is $W \cong \mathrm{Sym }(n)$.

The representation $\rho$ in~\eqref{equ:elements-H} has three
irreducible components, of dimensions $r_1$, $r_2$ and~$r_3$
respectively.  The inverses of the dominant weights for their
contragredient representations take the following values at
$t = (\diag(a_1,\ldots,a_n),\lambda) \in \aT_0$:
\[
\omega_1^{\, -1}(t) = \lambda a_1^{\, -(m-1)}, \qquad \omega_2^{\,
  -1}(t) = \lambda^{-1} a_n^{\, m}, \qquad \omega_3^{\, -1}(t) = a_n.
\]
In particular, we observe that $\omega_3^{\, -1} = \omega_1^{\, -1}
\omega_2^{\, -1} (\alpha_1 \alpha_2 \cdots \alpha_{n-1})^{m-1}$.
Furthermore, the identities
\[
a_i = (\omega_3^{\, -1} \alpha_{n-1} \alpha_{n-2} \cdots \alpha_i)(t),
\quad \text{for $1 \leq i \leq n$,} \qquad \text{and} \qquad \lambda =
\omega_1^{-1}(t) a_1^{\, m-1}
\]
show that
\[
\alpha_1, \ldots, \alpha_{n-1} \qquad \text{and} \qquad \alpha_0
\coloneqq \omega_1^{\, -1}, \quad \alpha_n \coloneqq \omega_2^{\,
  -1}
\]
form a basis for $\Hom(\aT_0,\aGm)$.  We define $\xi_0, \ldots,
\xi_n \in \Hom(\aGm,\aT_0)$ to be the dual basis.  This implies,
for $\tau \in \aGm$,
\begin{align*}
  \xi_i(\tau) & = \big( \diag(\underbrace{\tau^m,\ldots,\tau^m}_{i
    \text{ entries}},\underbrace{\tau^{m-1},\ldots,\tau^{m-1}}_{n-i
    \text{
      entries}}), \tau^{(m-1)m} \big), \quad \text{for $1 \leq i \leq n-1$,} \\
  \xi_0(\tau) & = \big(\diag(\tau,\ldots,\tau),\tau^m \big), \qquad
  \text{and} \qquad \xi_n(\tau) =
  \big(\diag(\tau,\ldots,\tau),\tau^{m-1} \big).
\end{align*}

As before, consider a number field~$\fk$.  For a finite prime $\fp$,
let $\fk_{\fp}$ denote the completion at $\fp$, and let $\fo_\fp$
denote the valuation ring of~$\fk_\fp$.  Fix a uniformising parameter
$\pi$ for $\fo_\fp$, and let $q$ denote the size of the residue field
$\fo_\fp / \pi \fo_\fp$.  We put
\[
H_0 = \aH_0(\fk_\fp) \qquad \text{and} \qquad T_0 = \aT_0(\fk_\fp).
\]

Writing $\xi = \prod_{i=0}^n \xi_i^{\, e_i}$, we obtain
\begin{equation*} 
  \xi(\pi) = (A,\lambda) = \big( \diag
  (a_1, \ldots, a_n),\lambda \big) \in T_0, 
\end{equation*}
where
\begin{equation} \label{equ:long-formula}
  \begin{split}
    a_i & = \pi^{(m-1) \left( \sum_{l=1}^{i-1} e_l \right) + m
      \left( \sum_{l=i}^{n-1} e_l \right) + (e_0 +
      e_n)}, \quad \text{for $1 \leq i \leq n$,} \\
    \lambda & = \pi^{m(m-1) \left( \sum_{l=1}^{n-1} e_l \right)
      +me_0+(m-1)e_n}.
  \end{split}
\end{equation}
We observe that, for non-negative exponents $e_i$, $i \in
\{1,\ldots,n-1\}$, the $\fp$-adic absolute values of the first $n$
diagonal entries of $\xi(\pi)$ are increasing: $\lvert a_1 \rvert_\fp
\leq \ldots \leq \lvert a_n \rvert_\fp$.  Furthermore, we obtain
\begin{align*}
  \det \big( \rho(\xi(\pi)) \big) & = \lambda^{r_1} \det ( \rho_1(A))
  \cdot \lambda^{-r_2} \det( \rho_2(A) ) \cdot \det(A) \\
  & = \lambda^{r_1} \left( \det(A) \right)^{- (m-1)r_1/n} \cdot
  \lambda^{- r_2}  \left( \det(A) \right)^{mr_2/n} \cdot \det(A) \\
  & = \lambda^{- \binom{m+n-2}{n-2}} \left( \det(A) \right)^{1 +
    \binom{m+n-2}{n-1}},
\end{align*}
where
\[
\det(A) = \prod\nolimits_{i=1}^n a_i = \pi^{n(m-1) \left(
    \sum_{l=1}^{n-1} e_l \right) + \left( \sum_{l = 1}^{n-1}
      l e_l \right) + n(e_0 + e_n)}.
\]
Altogether we obtain
\begin{multline*}
  \det \big( \rho(\xi(\pi)) \big) = \pi^{-\binom{m+n-2}{n-2} \left[
      m(m-1) \left( \sum_{l = 1}^{n-1} e_l \right) + me_0 +
      (m-1)e_n \right]} \\
  \cdot \pi^{\left( 1 + \binom{m+n-2}{n-1} \right) \left[ n(m-1)
      \left( \sum_{l=1}^{n-1} e_l \right) + \left( \sum_{l =
          1}^{n-1} l e_l \right) + n(e_0 + e_n) \right]},
\end{multline*}
where the coefficient of $e_i$ for $1 \leq i \leq n-1$ in the exponent
is
\begin{align}
  & \text{coeff.\ of $e_i$:} & & -m (m-1) \tbinom{m+n-2}{m} +
                                 \left( 1 + \tbinom{m+n-2}{m-1}
                                 \right) \big( (m-1)n+i \big)
                                 \label{equ:coeff-i}%
                                 \intertext{and the coefficients of
                                 $e_0$ and $e_n$ are}%
  & \text{coeff.\ of $e_0$:} & & \quad n \left( 1 + \tbinom{m+n-2}{m-1}  
                                 \right) - m \tbinom{m+n-2}{m} =
                                 \tbinom{m+n-2}{m-1} + n,
                                 \label{equ:coeff-n} \\
  & \text{coeff.\ of $e_n$:} & & \quad n \left( 1 +
                                 \tbinom{m+n-2}{m-1} \right) -
                                 (m-1) \tbinom{m+n-2}{m} =
                                 \tbinom{m+n-1}{m} + n. \label{equ:coeff-n+1}
\end{align}

Next we work out the term $\theta(\rho(\xi(\pi)))$ which appears in
Proposition~\ref{proposition:formula-integral}, using
$\theta_1,\theta_2$.  As we already noticed,
by~\eqref{equ:long-formula}, it is enough for our purposes to consider
elements
\begin{equation} \label{equ-t-A-lambda} t = (A,\lambda) = \big(
  \diag(a_1,\ldots,a_n), \lambda \big) \in T_0 \qquad \text{with
    $\lvert a_1 \rvert_\fp \leq \ldots \leq \lvert a_n \rvert_\fp$.}
\end{equation}
Using this restriction and writing
$v_\fp(\cdot) = - \log_q \lvert \cdot \rvert_\fp$, we obtain from the
description of the unipotent radical in
Theorem~\ref{thm:alg-autom-grp},
\begin{equation} \label{equ:log-theta-1} \log_q (\theta_1(\rho(t))) =
  \sum_{k=1}^n \binom{m+k-2}{k-1} \sum_{\mathbf{f} \in \mathbf{F}(k)}
  v_\fp \left( \lambda^{-1} \prod\nolimits_{i=k}^n a_i^{\, f_i}
  \right),
\end{equation}
where
\begin{equation*}
  \mathbf{F}(k) = \{ \ff \mid \ff = (f_1,\ldots,f_n) \in \mathbf{F}
  \text{ with } f_1 = \ldots = f_{k-1} = 0, f_k \geq 1 \}, \quad
  \text{for $1 \leq k \leq n$}.
\end{equation*}
This formula can be justified as follows.  Referring to the notation
in Theorem~\ref{thm:alg-autom-grp}, the entries of the matrix $C =
(b_{\ee + \ff})_{(\ee,\ff) \in \mathbf{E} \times \mathbf{F}}$ are
indexed by the elements of
\[
\mathbf{G} = \{ \gg \mid \gg = (g_1, \ldots, g_n) \in \N_0^n \text{
  with $g_1 + \ldots + g_n = 2m-1$} \}.
\]
For $\gg = (g_1,\ldots,g_n) \in \mathbf{G}$, the value $b_\gg$ occurs
in the column labeled by $\ff = (f_1,\ldots,f_n) \in \mathbf{F}$ if
and only if $\ff \leq \gg$, i.e., $f_1 \leq g_1, \ldots, f_n \leq
g_n$.  The $(\ff,\ff)$-entry of the diagonal matrix $\lambda^{-1}
\rho_2(A)$ is equal to $\lambda^{-1} \prod\nolimits_{i=1}^n a_i^{\,
  f_i}$.  This yields the condition
\[
v_\fp(b_\gg) \geq \max \left\{ - v_\fp \left( \lambda^{-1}
    \prod\nolimits_{i=1}^n a_i^{\, f_i} \right) \mid \ff \in
  \mathbf{F} \text{ with } \ff \leq \gg \right\}.
\]
We observe that the maximum on the right-hand side is attained for
$\ff \leq \gg$ such that $\ff$ is minimal with respect to the
lexicographical (non-cyclical) ordering on $\mathbf{F}$; this follows
from the restriction~\eqref{equ-t-A-lambda}.  We arrange the summation
over $\mathbf{F} = \mathbf{F}(1) \sqcup \ldots \sqcup \mathbf{F}(n)$
rather than $\mathbf{G}$.  For $1 \leq k \leq n$ and
$\mathbf{f} \in \mathbf{F}(k)$,
\[
\# \Big\{ \gg \in \mathbf{G} \mid \ff \leq \gg \,\text{ and }\, \neg \,
\exists \ff' \in \mathbf{F}(k) \cup \ldots \cup
\mathbf{F}(n): \ff' \leq \gg \,\land\, f_k'<f_k \Big\}
\]
is equal to the number of ways of distributing $m-1$ increments among
the first $k$ coordinates, giving $\binom{m-1+k-1}{k-1} =
\binom{m+k-2}{k-1}$.  In this way we obtain the summands
\[
\binom{m+k-2}{k-1} \, v_\fp \left( \lambda^{-1} \prod\nolimits_{i=k}^n
  a_i^{f_i} \right), \qquad \text{for each $\mathbf{f} \in \mathbf{F}(k)$},
\]
in the formula~\eqref{equ:log-theta-1}.

Observe that
$\lvert \mathbf{F}(k) \rvert = \binom{m+(n-k+1)-1}{(n-k+1)-1} -
\binom{m + (n-k) -1}{(n-k)-1} = \binom{m+n-k-1}{n-k}$,
for $1 \leq k \leq n$.  Writing
$N(n,m) = \sum_{\ff \in \mathbf{F}} f_1 = \sum_{\ff \in \mathbf{F}(1)}
f_1 = \frac{m}{n} \binom{m+n-1}{n-1} = \binom{m+n-1}{n}$,
we deduce from~\eqref{equ:log-theta-1} that
\begin{align*}
  \log_q(\theta_1(\rho(t))) & = \sum\nolimits_{k=1}^n \tbinom{m+k-2}{k-1}
  \sum\nolimits_{\mathbf{f} \in \mathbf{F}(k)} v_\fp \left(
    \lambda^{-1} \prod\nolimits_{i=k}^n a_i^{f_i} \right) \\
  & = \sum\nolimits_{k=1}^n \tbinom{m+k-2}{k-1} \lvert \mathbf{F}(k)
  \rvert \, v_\fp (\lambda^{-1}) \\
  & \quad + \sum\nolimits_{i=1}^n \left( \sum\nolimits_{k=1}^i
    \tbinom{m+k-2}{k-1} \sum\nolimits_{\mathbf{f} \in
      \mathbf{F}(k)} f_i \right) \, v_\fp(a_i) \\
  & = \sum\nolimits_{k=1}^n \tbinom{m+k-2}{k-1} \tbinom{m+n-k-1}{n-k}
  \, v_\fp (\lambda^{-1}) \\
  & \quad + \sum\nolimits_{i=1}^n \Big( \tbinom{m+i-2}{i-1}
  N(n-i+1,m) \\
  & \quad \quad + \sum\nolimits_{k=1}^{i-1} \tbinom{m+k-2}{k-1} \big(
  N(n-k+1,m) - N(n-k,m) \big) \Big) \, v_\fp(a_i) \\
  & = \tbinom{2m+n-2}{n-1} \, v_\fp (\lambda^{-1}) \\
  & \quad + \sum\nolimits_{i=1}^n \Big( \tbinom{m+i-2}{i-1}
  \tbinom{m+n-i-1}{n-i} + \sum\nolimits_{k=1}^i \tbinom{m+k-2}{k-1}
  \tbinom{m+n-k-1}{n-k+1} \Big) \, v_\fp(a_i).
\end{align*}
Using \eqref{equ:long-formula} and writing
\begin{equation*}
  R(m,n) = \sum\nolimits_{i=1}^n \Big( \tbinom{m+i-2}{i-1}
  \tbinom{m+n-i-1}{n-i} + \sum\nolimits_{k=1}^i \tbinom{m+k-2}{k-1}
  \tbinom{m+n-k-1}{n-k+1} \Big) 
  - m \tbinom{2m+n-2}{n-1},
\end{equation*}
we deduce that
\begin{align*}
  \log_q & \big(\theta_1(\rho(\xi(\pi))) \big) \\
  & = - \tbinom{2m+n-2}{n-1} \left( m(m-1) \left( \sum\nolimits_{l =
        1}^{n-1} e_l \right)  + m e_0 + (m-1) e_n \right) \\
  & \quad + \sum\nolimits_{i=1}^n \left( \tbinom{m+i-2}{i-1}
    \tbinom{m+n-i-1}{n-i} + \sum\nolimits_{k=1}^i \tbinom{m+k-2}{k-1}
    \tbinom{m+n-k-1}{n-k+1} \right) \\
  & \quad \quad \cdot \left( (m-1) \left( \sum\nolimits_{l=1}^{n-1}
      e_l \right) + \left( \sum\nolimits_{l = i}^{n-1} e_l
    \right) + (e_0 + e_n) \right) \\
  & = \sum\nolimits_{j=1}^{n-1} \Big( (m-1) R(m,n) +
  \sum\nolimits_{i=1}^j \Big( \tbinom{m+i-2}{i-1}
  \tbinom{m+n-i-1}{n-i} \\
  & \quad \quad + \sum\nolimits_{k=1}^i \tbinom{m+k-2}{k-1}
  \tbinom{m+n-k-1}{n-k+1} \Big) \Big) e_j 
  + R(m,n) e_0 + \Big( R(m,n) + \tbinom{2m+n-2}{n-1} \Big) e_n.
\end{align*}
We observe that, in fact, 
\begin{align*}
  R(m,n) & = \sum\nolimits_{i=1}^n \tbinom{m+i-2}{i-1}
           \tbinom{m+n-i-1}{n-i} +
           \sum\nolimits_{k=1}^n \sum\nolimits_{i=k}^n
           \tbinom{m+k-2}{k-1} 
           \tbinom{m+n-k-1}{n-k+1}  - m \tbinom{2m+n-2}{n-1}\\
         & = \tbinom{2m+n-2}{n-1} +
           \sum\nolimits_{k=1}^n (n-k+1)
           \tbinom{m+k-2}{k-1} 
           \tbinom{m+n-k-1}{n-k+1}  - m \tbinom{2m+n-2}{n-1}\\
         & = (m-1) \sum\nolimits_{k=1}^n 
           \tbinom{m+k-2}{k-1} 
           \tbinom{m+n-k-1}{n-k}  - (m-1) \tbinom{2m+n-2}{n-1}\\
         & = 0.
\end{align*}
Thus we obtain
\[
\log_q \big(\theta_1(\rho(\xi(\pi))) \big) = \sum\nolimits_{i=1}^{n-1}
C_i(m,n) e_i + \tbinom{2m+n-2}{n-1} e_n,
\]
where
\begin{align*}
  C_i(m,n) & = 
             \sum\nolimits_{j=1}^i \Big(
             \tbinom{m+j-2}{j-1} \tbinom{m+n-j-1}{n-j} + \sum\nolimits_{k=1}^j
             \tbinom{m+k-2}{k-1} \tbinom{m+n-k-1}{n-k+1} \Big) \\
           & = \sum\nolimits_{j=1}^i \tbinom{m+j-2}{m-1} \tbinom{m+n-j-1}{m-1} +
             \sum\nolimits_{k=1}^i (i-k+1) \tbinom{m+k-2}{m-1}
             \tbinom{m+n-k-1}{m-2} \\
           &  = \sum\nolimits_{j=1}^i \Big( 1 + \tfrac{(m-1)(i-j+1)}{n-j+1}
             \Big) \tbinom{m+j-2}{m-1} \tbinom{m+n-j-1}{m-1}.
\end{align*}
Hence we arrive at
\begin{equation} \label{equ:theta-1} \log_q
  \big(\theta_1(\rho(\xi(\pi)))\big) = \sum\nolimits_{i=1}^{n-1} \Big(
  \sum\nolimits_{j=1}^i \Big( 1 + \tfrac{(m-1)(i-j+1)}{n-j+1} \Big)
  \tbinom{m+j-2}{m-1} \tbinom{m+n-j-1}{m-1} \Big) e_i +
  \tbinom{2m+n-2}{n-1} e_n.
\end{equation}


This finishes our computation of~$\theta_1(\rho(\xi(\pi)))$.  It is much
easier to determine~$\theta_2(\rho(\xi(\pi)))$.  Indeed, from the
description of the unipotent radical in
Theorem~\ref{thm:alg-autom-grp} we observe that, for $t = \big(
\diag(a_1,\ldots,a_n), \lambda \big) \in T_0$,
\[
\log_q(\theta_2(\rho(t))) = (r_1 + r_2) \sum\nolimits_{i=1}^n v_\fp(a_i) =
\Big( \tbinom{m+n-2}{n-1} + \tbinom{m+n-1}{n-1} \Big)
\sum\nolimits_{i=1}^n v_\fp(a_i),
\]
and hence, using \eqref{equ:long-formula},
\begin{multline} \label{equ:theta-2} \log_q
  \big(\theta_2(\rho(\xi(\pi))) \big) = \Big( \tbinom{m+n-2}{m-1} +
  \tbinom{m+n-1}{m} \Big) \\
  \cdot \left( \big( (m-1)n + l \big) \left( \sum\nolimits_{l=1}^{n-1}
      e_l \right) + n(e_0 + e_n) \right).
\end{multline}

\smallskip

We are ready to apply Proposition~\ref{proposition:formula-integral}
to the group~$\aH \cong_\Z \aH_0$.  Writing $\beta_0 = \prod_{\beta
  \in \Phi^+} \beta$ and denoting by $\ell$ the standard length
function on the Weyl group $W \cong \mathrm{Sym}(n)$, we
obtain
\begin{align*}
  \mathcal{Z}_{\aH,\fp,\theta}(s) & = \sum_{w \in W} q^{-\ell(w)} \,
  \sum_{\xi \in w \Xi_w^+} q^{\langle \beta_0, \xi \rangle} \, \lvert
  \det \big(\rho(\xi(\pi)) \big) \rvert^s \,
  \theta_1\big(\rho(\xi(\pi)) \big) \, \theta_2
  \big(\rho(\xi(\pi)) \big) \\
  & = \frac{\sum_{w \in W} q^{-\ell(w)} \prod_{i=1}^{n-1} X_i^{\,
      \nu_i(w)}}{\prod_{i=1}^{n-1} (1-X_i)} \cdot
  \frac{1}{(1-\widetilde{X}_0)(1-\widetilde{X}_n)},
\end{align*}
where 
\[
\nu_i(w) = 
\begin{cases}
  1 & \text{if $\alpha_i \in w(\Phi^-)$,} \\
  0 & \text{otherwise},
\end{cases}
\]
each $X_i$, for $1 \le i \le n-1$, accounts for the
`$e_i$-contributions' of $\xi = \prod_{i=0}^n \xi_i^{\, e_i}$ and
$\widetilde{X}_0, \widetilde{X}_n$ account for the contributions of
$e_0,e_n$, as specified below.

We observe that $\beta_0 = \prod_{i=1}^{n-1} \alpha_i^{\, i(n-i)}$,
hence $\langle \beta_0, \xi \rangle = \sum_{i=1}^{n-1} i(n-i)e_i$.
Moreover, for $1 \leq i \leq n-1$,
\begin{align*}
  \log_q(X_i) & = i(n-i) + \underbrace{\Big( \tbinom{m+n-2}{m-1} +
                \tbinom{m+n-1}{m} \Big) \big( (m-1)n + i
                \big)}_{\text{contributions from 
                $\theta_2$ according to \eqref{equ:theta-2}}}
  \\
              & \quad + \underbrace{\sum\nolimits_{j=1}^i \Big( 1 +
                \tfrac{(m-1)(i-j+1)}{n-j+1} 
                \Big) \tbinom{m+j-2}{m-1} \tbinom{m+n-j-1}{m-1}
                }_{\text{contributions from
                $\theta_1$ according to \eqref{equ:theta-1}}} \\
              & \quad - \underbrace{\Big( - m(m-1) \tbinom{m+n-2}{m}
                + \Big( 1 + \tbinom{m+n-2}{m-1} \Big) \big( (m-1)n + i
                \big) \Big)}_{\text{contributions from
                the determinant according to \eqref{equ:coeff-i}}} s
\end{align*}  
and
\begin{align*}
  \log_q(\widetilde{X}_0) & = \underbrace{n \Big( \tbinom{m+n-2}{m-1} +
                            \tbinom{m+n-1}{m} \Big)}_{\text{contributions from
                            \eqref{equ:theta-2}}} - \underbrace{\Big(
                            \tbinom{m+n-2}{m-1} + n \Big) s}_{\text{contributions from
                            \eqref{equ:coeff-n}}}, \\
  \log_q(\widetilde{X}_n) & = \underbrace{n \Big( \tbinom{m+n-2}{m-1} +
                            \tbinom{m+n-1}{m} \Big) +
                            \tbinom{2m+n-2}{2m-1}}_{\text{contributions from
                            \eqref{equ:theta-2} and 
                            \eqref{equ:theta-1}}}
                            - \underbrace{\Big( \tbinom{m+n-1}{m} + n \Big)
                            s }_{\text{contributions from 
                            \eqref{equ:coeff-n+1}}}.
\end{align*}
In view of Remark~\ref{remark:Lie-correspondence}, this completes the
proof of Theorem~\ref{theorem:main}.  


\section{Local functional equations, abscissae of convergence and
  meromorphic continuation} \label{sec:corollaries}

In this section we analyse the explicit formulae in
Theorem~\ref{theorem:main} to deduce Corollaries~\ref{corollary:fn-eq},
\ref{corollary:mero-cont} and~\ref{cor:limit-points}. 

\begin{proof}[Proof of Corollary~\ref{corollary:fn-eq}]
  The formula we obtained in Section~\ref{sec:explicit-computation}
  for $\mathcal{Z}(s) = \mathcal{Z}_{\aH,\fp,\theta}(s)$ is classical
  and was treated by Igusa~\cite{Ig89}; also
  see~\cite[Section~2]{KlVo09}.  We follow the adapted approach given
  by du Sautoy and Lubotzky~\cite[p.~82]{dSLu96}.  Let $w_0$ denote
  the unique element of maximal length $\ell(w_0) = \binom{n}{2}$ in
  the Weyl group~$W \cong \mathrm{Sym}(n)$.  It is well-known that
  $\ell(w)+\ell(ww_0) = \lvert \Phi^+ \rvert$ and
  $w_0(\Phi^-) = \Phi^+$ for all $w\in W$; see, for instance,
  \cite[Section~1.8]{Hu90}.  The latter implies that
  $\nu_i(w)=1-\nu_i(ww_0)$ for $w \in W$ and $i \in \{1,\ldots, n-1\}$.
  Hence
  \begin{align*}
    \mathcal{Z}(s) \vert_{q\to q^{-1}} & = \frac{\sum_{w \in
                                         W} q^{\ell(w)} (-1)^{n-1}\prod_{i=1}^{n-1} X_i^{\,
                                         1-\nu_i(w)}}{\prod_{i=1}^{n-1} (1-X_i)} \cdot
                                         \frac{\widetilde{X}_0\widetilde{X}_n}{(1-\widetilde{X}_0)(1-\widetilde{X}_n)}\\ 
                                       &=(-1)^{n-1}q^{\lvert \Phi^+ \rvert}\frac{\sum_{w \in W}
                                         q^{-\ell(ww_0)} \prod_{i=1}^{n-1} X_i^{\,
                                         \nu_i(ww_0)}}{\prod_{i=1}^{n-1} (1-X_i)} \cdot
                                         \frac{\widetilde{X}_0\widetilde{X}_n}{(1-\widetilde{X}_0)(1-\widetilde{X}_n)}\\
                                       &=(-1)^{n-1}q^{\lvert \Phi^+ \rvert}\widetilde{X}_0
                                         \widetilde{X}_n\mathcal{Z}(s).
  \end{align*}
  Noting that $\lvert \Phi^+ \rvert = \binom{n}{2}$, we obtain a
  symmetry factor of $q^{a+bs}$, where the parameters $a$ and $b$ are
  given by~\eqref{equ:sym-fac-a-b}.  The special case $\fk=\Q$ yields
  the functional equation for the local
  factors~$\zeta^\wedge_{\Delta_{m,n},p}(s)$.
\end{proof}

In preparation for the proof of Corollary~\ref{corollary:mero-cont} we
derive two lemmata.

\begin{lemma} \label{lem:maximum-A-B}
  Let $A_i, B_i$, for $0 \le i \le n$, and
  $\widetilde{A}_j, \widetilde{B}_j$, for $j \in \{0,n\}$ be as in
  Theorem~\textup{\ref{theorem:main}}.   Suppose that $m \ge 2$.  Then
  \[ 
  \max \left\{ \frac{A_i+1}{B_i} \mid 1 \le i \le n-1 \right\} <
  \max \left\{ \frac{\widetilde{A}_0+1}{\widetilde{B}_0},
    \frac{\widetilde{A}_n+1}{\widetilde{B}_n} \right\}.
  \]
\end{lemma}

\begin{proof}
  We observe that
  \[
  \frac{\widetilde{A}_0 +1}{\widetilde{B}_0} = \frac{A_0 + (m-1)}{B_0}
  \ge \frac{A_0+1}{B_0}
  \]
  and similarly $(\widetilde{A}_n+1)/\widetilde{B}_n \ge
  (A_n+1)/B_n$.  Consequently it suffices to show that
  \[
  \max \left\{ \frac{A_i+1}{B_i} \mid 1 \le i \le n-1 \right\} < \max
  \left\{ \frac{A_0+1}{B_0},  \frac{A_n+1}{B_n} \right\}.
  \]
  Observe that for any positive numbers $a,b,x,y$ we have
  \begin{equation}\label{equ:basic-inequ}
    \frac{x}{y} \ge \frac{a+x}{b+y} \ge \frac{a}{b} \quad \Longleftrightarrow \quad
    \frac{x}{y} \ge \frac{a}{b}, 
  \end{equation}
  and similarly with strict inequalities.  Writing
  $x_i = (A_i+1) - (A_{i-1}+1) = A_i - A_{i-1}$ and
  $y_i = B_i - B_{i-1}$ for $1 \le i \le n$, we deduce that it
  suffices to prove that
  \[
   0 < x_1/y_1 < \ldots < x_n/y_n.
  \] 
  For then \eqref{equ:basic-inequ} implies that $(A_i+1)/B_i$,
  $0 \le i \le n$, forms a convex sequence and takes its maximum at
  $i=0$ or $i=n$.

  In fact, $y_i = 1 + \binom{m+n-2}{m-1}$ is constant and, for $1 \le
  i \le n$,
  \begin{multline*}
    x_i = \Big( \tbinom{m+n-2}{m-1} + \tbinom{m+n-1}{m} + n + 1 \Big)
    -2i \\
    + \tbinom{m+i-2}{m-1} \tbinom{m+n-i-1}{m-1} + \sum\nolimits_{j=1}^i
    \tfrac{m-1}{n-j+1} \tbinom{m+j-2}{m-1} \tbinom{m+n-j-1}{m-1}.
  \end{multline*}
  Clearly, $x_1 \geq n-1 > 0$, and it remains to check that $d_i = x_i - x_{i-1}$
  is positive for $2 \le i \le n$.  Indeed, we obtain
  \begin{align*}
    d_i & = -2 + \left( \frac{m+n-i}{n-i+1} -
          \frac{(i-1)(m+n-i)}{(m+i-2)(n-i+1)} \right)
          \binom{m+i-2}{m-1} \binom{m+n-i-1}{m-1} \\
        & = -2 + \frac{(m+n-i) (m-1)}{(n-i+1) (m+i-2)} \,
          \frac{(m+i-2)!}{(m-1)! (i-1)!} \,
          \frac{(m+n-i-1)!}{(m-1)!(n-i)!} \\
        & = -2 + \frac{(m+i-3) (m+i-4) \cdots i}{(m-2)!} \, \frac{(m+n-i)(m+n-i-1)
          \cdots (n-i+2)}{(m-1)!} \\
        & \ge -2 + i (n-i+2) \\
        & > 0. \qedhere
  \end{align*}
\end{proof}

\begin{lemma} \label{lem:appendix-inequ} Let
  $\widetilde{A}_j, \widetilde{B}_j$, for $j \in \{0,n\}$ be as in
  Theorem~\textup{\ref{theorem:main}}, and let $\mathcal{C}$ be as in
  Corollary~\ref{corollary:mero-cont}.  Suppose that $m \ge 2$.  Then
  \[ 
  \frac{\widetilde{A}_0+1}{\widetilde{B}_0} <
  \frac{\widetilde{A}_n+1}{\widetilde{B}_n} \qquad \Longleftrightarrow
  \qquad \frac{\widetilde{A}_0+1}{\widetilde{B}_0} \le
  \frac{\widetilde{A}_n+1}{\widetilde{B}_n} \qquad \Longleftrightarrow
  \qquad (m,n) \not\in \mathcal{C}.
    \]
\end{lemma}

\begin{proof}
  Setting
  \[
  F(m,n) = \frac{\binom{2m+n-2}{2m-1}}{\binom{m+n-2}{m}} -
  \frac{\binom{m+n-2}{m-1} \left( \frac{n(n-1)}{m} + 2n \right) +
      1}{\binom{m+n-2}{m-1} + n},
  \]
  we observe from~\eqref{equ:basic-inequ} and the formulae in
  Theorem~\textup{\ref{theorem:main}} that
  \[
  \frac{\widetilde{A}_0+1}{\widetilde{B}_0} \le
  \frac{\widetilde{A}_n+1}{\widetilde{B}_n} \quad \Longleftrightarrow
  \quad \frac{\widetilde{A}_n - \widetilde{A}_0}{\widetilde{B}_n -
    \widetilde{B}_0} \ge \frac{\widetilde{A}_0 +1}{\widetilde{B}_0}
  \quad \Longleftrightarrow \quad F(m,n) \geq 0,
  \]
  and similarly with strict inequalities.  A standard calculation shows
  that
  \[
  F(m,n) = \frac{f_m(n)}{\frac{(2m-1)!}{(m-1)!} (n-1) n \left(
      \prod_{i=1}^{m-2} (n+i) + (m-1)!\right)},
  \]
  where 
  \begin{multline*}
    f_m = m t \left( \prod_{i=m-1}^{2m-2} (t+i) \right) \left(
      \prod_{i=1}^{m-2} (t+i) + (m-1)! \right) \\ - (t-1) \left(
      \prod_{i=m+1}^{2m-1} i \right) \left( t^2
      (t+2m-1)\prod_{i=1}^{m-2} (t+i) + m! \right) \in \mathbb{Z}[t].
  \end{multline*}
  We need to prove that $f_m(n) \ge 0$ if
  and only if $(m,n) \not\in \mathcal{C}$, and that in these cases
  even the strict inequality $f_m(n) > 0$ holds.  For $1 \leq m \leq 6$, we
  compute $f_m$ explicitly and a routine analysis yields the desired
  result; see Table~\ref{tab:1}.

  \begin{table}[htb!]
    \centering
    \caption{The polynomials $f_m$ for $1 \le m \le 6$}
    \label{tab:1}
    \def\arraystretch{1.05}
   \begin{tabular}{|c||l|l|}
      \hline
      $m$ & $f_m$ & comments \\
      \hline\hline
      $2$ & $-3t^4 - 2t^3 + 21t^2 + 2t + 6$ & 
                                              $f_2(2) = 
                                              30$,
                                              $f_2(3)
                                              =
                                              -96$,
      \\
          & &  $f_2'(x) < 0$
              for
              $x
              \ge
              2$ \\
      \hline
      $3$ & $-17t^5 - 64t^4 + 179t^3 + 406t^2 + 96t + 120$ & 
                                                             $f_3(2) = 
                                                             1800$,
                                                             $f_3(3)
                                                             =
                                                             -420$,
      \\
          & &  $f_3'(x) < 0$
              for
              $x
              \ge
              3$ \\
      \hline
      $4$ & $4 t^7 - 126 t^6 - 1166 t^5 + 642 t^4 + 11242 t^3$ &
                                                                 $f_4(n)
                                                                 > 0$
                                                                 for
                                                                 $n
                                                                 \in \{2,3,39\}$,
      \\
          & $\quad + 18204
            t^2 + 6480 t + 
            5040$ & $f_4(n)
                    < 0$
                    for
                    $n
                    \in
                    \{4,5,\ldots,38\}$, \\ 
          & & $f_4'(x) > 0$ for $x \ge 34$ \\
      \hline
      $5$ & $5 t^9 + 180 t^8 - 294 t^7 - 19536 t^6 - 35355 t^5$ & $f_5(n)
                                                                  > 0$
                                                                  for
                                                                  $n
                                                                  \in \{2,3,4,10\}$,\\
          & $\quad  + 258060 t^4 + 993244 t^3
            + 1424496 t^2$ & $f_5(n)
                             < 0$
                             for
                             $n
                             \in
                             \{5,6,7,8,9\}$, \\
          & $\quad + 645120 t + 362880$ & $f_5'(x) > 0$ for $x \ge 9$
      \\
      \hline
      $6$ & $6 t^{11} + 330 t^{10} + 7920 t^9 + 53460 t^8 - 161442
            t^7$ & $f_6(2) 
                   > 0$,\\
          & $\quad  - 1429830 t^6 + 
            5025180 t^5 + 48636840 t^4$ & $f_6'(x)
                                          > 0$
                                          for
                                          $x \ge 2$, \\
          & $\quad + 125765136 t^3 + 170467200 t^2 + 90720000 t $ &
      \\
          & $\quad + 39916800$ & 
      \\
      \hline
    \end{tabular}
  \end{table}

  To conclude the proof it suffices to show that, for $m \ge 7$, the
  polynomial $f_m$ has non-negative coefficients.  For $7 \le m \le 29$
  this can be checked directly.  Now suppose that $m \ge 30$.  A short
  calculation reveals that
  \[
  f_m = g_m +  \left( \prod_{i=0}^{m-2} (t+i) \right) h_m,
  \]
  where
  \begin{align*}
  g_m & = m! \, t \prod_{i=m-1}^{2m-2} (t+i) - (2m-1)! \, (t-1), \\
  h_m & = m \prod_{i=m-1}^{2m-2} (t+i) - \tfrac{(2m-1)!}{m!} (t-1)t(t+2m-1).
  \end{align*} 
  Clearly, it is enough to prove that $g_m$ and $h_m$ have
  non-negative coefficients.  As for $g_m$, we only have to examine
  the coefficient of~$t$.  It is
  \[
  \left( m(m-1) - (2m-1) \right)
  (2m-2)! = (m^2 - 3m +1) (2m-2)! \ge 0
  \]
  for $m \geq 3$.  For $h_m$, we need to examine the coefficients of
  $t^3$ and $t^2$.  The coefficient of $t^3$ is non-negative if and
  only if
  \[
  m \sum_{\substack{a,b,c \text{ with} \\ m-1 \le a < b < c \le 2m-2}}
  \, \prod_{\substack{m-1 \le i \le 2m-2 \\\text{and } i \not =
      a,b,c}} i \ge \prod_{i=m+1}^{2m-1} i.
  \]
  Here, the left-hand side is at least
  \[   
  m \binom{m}{3} \prod_{i=m-1}^{2m-5} i = \frac{m^3 (m-1)^2 (m-2)}{6}
  \prod_{i=m+1}^{2m-5} i
  \]
  and it suffices to observe that
  \[
  \tfrac{1}{6} m^3 (m-1)^2 (m-2) \ge (2m-1)(2m-2)(2m-3)(2m-4)
  \]
  for $m \ge 10$.  Similarly, the coefficient of $t^2$ in $h_m$ is
  non-negative if and only if
  \[
  m \sum_{\substack{a,b \text{ with} \\ m-1 \le a < b \le 2m-2}}
  \, \prod_{\substack{m-1 \le i \le 2m-2 \\\text{and } i \not =
      a,b}} i \ge (2m-2) \prod_{i=m+1}^{2m-1} i.
  \]
  Here, the left-hand side is at least
  \[   
  m \binom{m}{2} \prod_{i=m-1}^{2m-4} i = \frac{m^3 (m-1)^2}{2}
  \prod_{i=m+1}^{2m-4} i
  \]
  and it suffices to observe that
  \[
  \tfrac{1}{2} m^3 (m-1)^2 \ge (2m-1)(2m-2)^2(2m-3)
  \]
  for $m \ge 30$.
\end{proof}

\begin{proof}[Proof of Corollary~\ref{corollary:mero-cont}] For $m=1$
  it is easy to check the claim directly; see
  Remark~\ref{rem:Grenham}.  Now suppose that $m \ge 2$.  Recall that
  $\beta(m,n) = \max \{ A_i / B_i \mid 1 \le i \le n-1 \}$, and for
  $w \in W$ let
  \[
  I(w) = \{ i \mid \text{$1 \le i \le n-1$ and
    $\alpha_i \in w(\Phi^-)$} \}
  \]
  denote the left descent set of $w$.  We use the Euler product
  decomposition
  $\zeta^\wedge_{\Delta_{m,n}}(s)=\prod_p
  \zeta^\wedge_{\Delta_{m,n},p}(s)$.
  Combining the formula \eqref{equ:main-thm} for the local zeta
  functions, obtained in Theorem~\ref{theorem:main}, with
  Lemmata~\ref{lem:maximum-A-B} and \ref{lem:appendix-inequ}, we see
  that it suffices to show that the `product of numerators'
  \[
  P(s) = \prod_p \Big( 1 + \sum_{w \in W \smallsetminus \{1\}}
  p^{-\ell(w)} \prod_{i \in I(w)} p^{A_i - B_i s} \Big)
  \]
  converges absolutely for $\mathrm{Re}(s) > \beta(m,n)$.  Note that
  the strict inequalities in the lemmata imply that we get a simple
  pole at $s=\alpha(m,n)$.

  Observe that an infinite product of complex numbers of the form
  $\prod_p(1+\sum_{j=1}^N z_{j,p})$ converges absolutely if and only
  if $\prod_p (1 + z_{j,p})$ converges absolutely for each
  $j \in \{1,\ldots,N\}$.  Therefore $P(s)$ converges absolutely if
  and only if
  \begin{equation*}
    P_w(s) = \prod_p \Big( 1 + 
    p^{-\ell(w)} \prod_{i \in I(w)} p^{A_i - B_i s} \Big) = \prod_p
    \left( 1 + p^{ -\ell(w) + \sum_{i \in 
          I(w)} (A_i -  B_i
        s)} \right)
  \end{equation*}
  converges absolutely for each $w \in W \smallsetminus \{1\}$.
  Clearly, $P_w(s)$ converges absolutely if and only if
  \[
  \mathrm{Re}(s) > \frac{1-\ell(w)+\sum_{i \in I(w)} A_i}{\sum_{i \in
      I(w)} B_i}.
  \]
  Noting that
  \begin{equation} \label{equ:A-B-max} \frac{1-\ell(w)+\sum_{i \in
        I(w)} A_i}{\sum_{i \in I(w)} B_i} \le \frac{\sum_{i \in I(w)}
      A_i}{\sum_{i \in I(w)} B_i} \le \max \left\{ \frac{A_i}{B_i}
      \mid i \in I(w)\right\},
  \end{equation}
  we deduce that $P(s)$ converges absolutely for
  $\mathrm{Re}(s) > \beta(m,n)$.  In fact, this method does not yield
  anything better, because each value $A_i/B_i$ is indeed attained
  with equalities in \eqref{equ:A-B-max} for a suitable simple
  reflection~$w$.
\end{proof}

\begin{proof}[Proof of Corollary~\ref{cor:limit-points}]
  We write $\widetilde{A}_n(m,n)$ and $\widetilde{B}_n(m,n)$ to
  indicate the dependence of $\widetilde{A}_n$ and $\widetilde{B}_n$
  on the parameters $m,n$.  From Corollary~\ref{corollary:mero-cont}
  we deduce that
  \begin{align*}
    \lim_{m\to \infty} \alpha(m,n) & = \lim_{m\to
                                     \infty} \big(\widetilde{A}_n
                                     (m,n)+1 \big)/\widetilde{B}_n (m,n) \\
                                   & = \lim_{m\to
                                     \infty} \frac{n \Big(
                                     \tbinom{m+n-2}{n-1} +
                                     \tbinom{m+n-1}{n-1} 
                                     \Big)+\tbinom{2m+n-2}{n-1}+1}{\tbinom{m+n-1}{n-1}
                                     + n}\\ 
                                   &= \lim_{m\to
                                     \infty} \frac{n \Big(
                                     \frac{m}{m+n-1}+1 
                                     \Big)\tbinom{m+n-1}{n-1}+\frac{2m+n-2}{m+n-1}\cdot
                                     \frac{2m+n-3}{m+n-2}\cdots\frac{2m}{m+1}
                                     \cdot\tbinom{m+n-1}{n-1}+1}    
                                     {\tbinom{m+n-1}{n-1} + n}\\
                                   & = 2n+2^{n-1}. \qedhere
  \end{align*}
\end{proof}


\end{document}